\documentclass[14pt]{article}
\usepackage{amsmath}
\usepackage{amssymb}
\usepackage{mathrsfs}
\usepackage{amscd}
\usepackage{amsthm,color}
\usepackage[all]{xy}
\SelectTips{cm}{}
\usepackage{authblk}
\usepackage{tikz}
\usepackage{dirtytalk}
\usepackage{stmaryrd}

\let\margin\marginpar
\newcommand\myMargin[1]{\margin{\raggedright\scriptsize #1}}
\renewcommand{\marginpar}[1]{\myMargin{#1}}

\newtheorem{lemma}{Lemma}[section]
\newtheorem{theorem}[lemma]{Theorem}
\newtheorem{corollary}[lemma]{Corollary}

\newtheorem{prop}[lemma]{Proposition}
\theoremstyle{definition}
\newtheorem{definition}[lemma]{Definition}

\newtheorem{remark}[lemma]{Remark}

\theoremstyle{remark}
\newtheorem*{proof*}{Proof}
\numberwithin{equation}{section}

\newcommand{\iso}{\stackrel{_\sim}{\rightarrow}}

\def\Spec{{\bf {Spec}}}

\def\D{\mathrm{D}}

\def\Ext{{\mathrm{Ext}}}

\def\Hom{{\mathrm{Hom}}}

\def\End{{\mathrm{End}}}

\def\Spec{{\mathrm{Spec\ }}}
\def\deg{{\mathrm{deg}}}

\def\Bimod{{-\mathrm{Bimod}}}

\def\PP{{\mathbb P}}

\def\ZZ{{\mathbb Z}}
\def\CC{{\mathbb C}}
\def\LL{{\mathbb L}}
\def\NN{{\mathbb N}}

\def\BB{{\mathbb B}}

\def\TT{{\mathbb T}}

\def\HoH{{\mathrm{HH}}}

\def\bk{{\bf{k}}}
\def\bw{{\bf{w}}}

\def\cG{{\mathcal{G}}}
\def\cF{{\mathcal{F}}}
\def\cE{{\mathcal{E}}}
\def\cO{{\mathcal{O}}}
\def\cC{{\mathcal{C}}}
\def\cL{{\mathcal{L}}}
\def\cM{{\mathcal{M}}}

\def\cA{{\mathcal{A}}}

\def\cT{{\mathcal{T}}}

\def\cD{{\mathcal{D}}}

\def\cJ{{\mathcal{J}}}
\def\cK{{\mathcal{K}}}
\def\cS{{\mathcal{S}}}

\def\fg{{\mathfrak{g}}}

\def\fm{{\mathfrak{m}}}
\def\fn{{\mathfrak{n}}}

\def\lg{{\langle}}
\def\rg{{\rangle}}
\def\lgg{{\langle\langle}}
\def\rgg{{\rangle\rangle}}
\def\ol{\overline}
\def\ot{{\otimes}}

\def\op{{\oplus}}
\def\wh{\widehat}
\def\wt{\widetilde}

\def\tr{{\mathrm{tr}}}

\def\deg{{\mathrm{deg}}}

\def\ep{{\epsilon}}

\def\Lm{{\Lambda}}

\def\p{{\prime}}
\def\pp{{\prime\prime}}
\def\1{{\bf{1}}}

\def\coh{{\mathrm{coh}}}

\def\cy{{\mathrm{cyc}}}
\def\Aut{{\mathrm{Aut}}}
\def\id{{\mathrm{Id}}}
\def\der{{\mathrm{Der}}}
\def\cder{{\mathrm{cDer}}}
\def\dder{{\mathrm{\mathbb{D}er}}}

\def\bx{{\bf{x}}}
\def\by{{\bf{y}}}
\def\im{{\mathrm{im}}}

\def\per{{\mathrm{per}}}

\def\tot{{\wt{\ot}}}
\def\bF{{\mathbf{F}}}
\def\bu{{\mathbf{u}}}
\def\bv{{\mathbf{v}}}

\def\ab{{\mathrm{ab}}}
\def\Rep{{\mathrm{Rep}}}
\def\gl{{\mathfrak{gl}}}
\def\ord{{\mathrm{ord}}}
\def\MH{{\mathrm{MH}}}
\def\QT{{\mathrm{QT}}}
\def\mod{{\mathfrak{mod}}}
\def\Hilb{{\mathrm{Hilb}}}
\def\Grass{{\mathrm{Grass}}}
\def\Ad{{\mathrm{Ad}}}

\title{Quivers with analytic potentials}
\date{}
\author{Zheng Hua\thanks{huazheng@maths.hku.hk}}
\affil{Department of Mathematics, the University of Hong Kong, Hong Kong SAR, China}
    \author{Bernhard Keller\thanks{bernhard.keller@imj-prg.fr}}
\affil{Universit\'e Paris Diderot -- Paris 7, 
    Sorbonne Universit\'e, 
    UFR de Math\'ematiques, 
    CNRS, 
   Institut de Math\'ematiques de Jussieu--Paris Rive Gauche, IMJ-PRG,
    B\^{a}timent Sophie Germain,
    75205 Paris Cedex 13,
    France
}

\begin{document}
\maketitle
\begin{abstract}
Given a quiver $Q$, a formal potential is called analytic if its coefficients are bounded 
by the terms of a geometric series. As shown by Toda,
the potentials appearing in the deformation theory of complexes of coherent sheaves on complex projective Calabi-Yau threefolds are analytic.   Our paper consists of two parts. In the first part, we establish the foundations of the differential calculus of quivers with analytic potentials and prove two fundamental results: the inverse function theorem and Moser's trick.  
As an application, we prove finite determinacy of analytic potentials with 
finite-dimensional Jacobi algebra, answering a question of Ben Davison. We also prove a Mather-Yau type theorem for analytic potentials with finite-dimensional Jacobi algebra, extending  previous work by the first author with Zhou.  In the second part, we study Donaldson-Thomas theory of quivers with analytic potentials. First, we construct a canonical perverse sheaf of vanishing cycles on the moduli stack of finite dimensional modules over the Jacobi algebra of an arbitrary iterated mutation via a separation lemma for analytic potentials. Finally, we  provide a transformation formula for DT invariants (weighted by the Behrend function) under iterated mutations as a counterpart of Nagao's result on topological DT invariants. Our result leads to a perverse analogue of 
the $F$-series in  cluster algebras. It also yields a Behrend--weighted 
Caldero--Chapoton formula for what we call anti-cluster algebras.
\end{abstract}

\newpage
\section{Introduction}
The differential calculus of quivers with potentials has been an important subject in noncommutative geometry since the work of Ginzburg \cite{Ginz} and Van den Bergh \cite{VdB15}. However, most research was done in the formal setting. Our paper is an attempt towards the analytic theory of quivers with potentials. This is not only 
motivated by theoretical curiosity but motivated by in cluster theory and algebraic geometry, where the notion of analytic potential appears naturally. Our results on analytic potentials also lead to solutions to some open problems in the Donaldson--Thomas theory of quivers with potentials.

Let $Q$ be a finite quiver. A (formal) potential is an infinite (complex) linear combination \[
\sum_{c}a_c\, c\; ,
\]
where $c$ runs through the cyclic paths of $Q$ and the $a_c$ are complex
numbers. The space of potentials $\wh{\CC Q}_\cy$ is defined to be   $\wh{\CC Q}/[\wh{\CC Q},\wh{\CC Q}]^{cl}$ where $\wh{\CC Q}$ is the complete path algebra and $[\wh{\CC Q},\wh{\CC Q}]^{cl}$ is the closure of the space of commutators. Potentials should be viewed as {\em formal} functions on a noncommutative space since we impose no growth condition on the coefficients $a_c$. Formal differential calculus of quivers with potentials has been studied extensively. In particular, we are inspired by the work of Derksen, Weyman and Zelevinsky \cite{DWZ}, where they establish the calculus of mutations of quivers with potentials. Their results have important applications in the study of cluster algebras, cf. \cite{DWZ2, Plamondon, Nagao}. In \cite{HuaZhou}, the first author and Zhou study the singularity theory of formal potentials. The results in \cite{HuaZhou} have applications to 3-dimensional birational geometry. 
For example, in \cite{HuaKeller}, we prove that the underlying singularity of a 
3-dimensional flopping contraction is determined by the derived Morita equivalence class of the noncommutative deformation algebra of the exceptional curve together with the class represented by the potential.

We call a potential \emph{analytic} if its coefficients are bounded by a geometric series, 
i.e. there is a constant $C>0$ such that $|a_c|\leq C^{|c|}$, where $|c|$ is the length of the cycle $c$. The differential calculus of analytic potentials is significantly harder than that of formal ones. The main difficulty is to control the convergence radius. Before stating our main results, let us briefly recall how 
analytic potentials appear in algebraic geometry and cluster algebras.

Let $X$ be a complex projective Calabi-Yau 3-fold. Let $\cE_1, \ldots, \cE_n$ be a collection of coherent sheaves such that $\dim \Hom(E_i,E_j)=\delta_{ij}$. We can associate to it a finite quiver $Q$ called the Ext-quiver of the collection. The set of nodes $Q_0$ is $\{1,2,\ldots,n\}$ and the number of arrows from $i$ to $j$ is equal to $\dim \Ext^1(E_i,E_j)$. By the Calabi-Yau property, one can produce
\cite{VdB15} a formal potential $\Phi$ on $Q$ which is unique up to right equivalence. The cyclic derivatives of $\Phi$ control the deformation-obstruction theory of $\bigoplus_{i=1}^n \cE_i$. Using ideas of Kuranishi and Fukaya, Toda proves that there exists an analytic representative in the right equivalence class of $\Phi$ \cite{Tod17}. We may simply assume that $\Phi$ itself is analytic. Let $\cE_1,\ldots,\cE_n$ be a collection of mutually non isomorphic stable sheaves of the same slope. The moduli stack of semi-stable sheaves near the polystable sheaf $\bigoplus_{i=1}^n\cE_i$ can be locally embedded into the moduli stack of finite dimensional representations of $Q$ as the critical stack of the trace of $\Phi$, which is called the \emph{Chern-Simons function} of the moduli stack (see \cite{Tod17}).

The link between Donaldson-Thomas theory and cluster algebras was first 
pointed out by Kontsevich and Soibelman \cite{KS08}. Given a 3-Calabi-Yau category 
satisfying suitable conditions, one can define two types of invariants: the topological and the weighted DT invariants. In the geometric context, one may consider $\D^b(\coh(X))$ for a projective CY 3-fold $X$. If we fix a stability condition, 
the moduli stack of semi-stable objects carries an intrinsic constructible function $\nu$,
the Behrend function (see \cite{Beh} \cite{JS08}). Roughly speaking, the topological and weighted DT invariants are defined respectively to be the topological and 
the $\nu$-weighted Euler characteristics of the moduli stack. DT-invariants admit various refinements where the Euler characteristic is replaced by suitable cohomological or motivic enhancements. In order to define the refined DT invariants, or to study the transformation of DT invariants under change of stability, one needs to prove 
the existence of Chern-Simons functions (see \cite{KS08} \cite{JS08}). 
They are analytic functions whose critical stacks give local models for the moduli stack.
The topological and the weighted DT invariants share some nice properties. 
For example, they both satisfy the wall crossing formula (see \cite{JS08}). However, a major difference is that only the weighted one is invariant under deformation
(cf. Corollary 5.28 of \cite{JS08}).

DT invariants also appear in a purely algebraic context.
Given a quiver with potential $(Q,\Phi)$, the derived category of the Ginzburg algebra provides another example of a 3CY--category fitting into the framework of 
Kontsevich--Soibelman. The topological DT invariants are defined to be the Euler characteristics of the moduli stacks of finite dimensional modules over the Jacobi algebra, or the quot schemes of certain modules over the Jacobi algebra. In \cite{Nagao}, Nagao studied the transformation of the topological DT invariants 
under mutation and used it to give an alternative proof for a series
of conjectures on cluster algebras by Fomin and Zelevinsky \cite{FominZelevinsky07}.
The refined DT invariants are closely related to quantum cluster algebras. 
Among other things, Nagao's work is  important because it provides another way to understand quantum cluster algebras. However, one major technical issue in this 
subject is that  mutations
of potentials with finitely many terms may have infinitely many terms. Therefore, unless we can control the growth of the coefficients of the mutated potential, one cannot expect the existence of Chern-Simons functions on the moduli stack of the mutated quiver with potential. This is the reason why Nagao was not able to deal with the weighted Euler characteristic. Analytic potentials provide a quite satisfactory solution to this problem.

The major results of this paper can be summarized as follows.
\paragraph{Noncommutative differential calculus}
We prove several theorems in the differential calculus of analytic potentials, which should have independent interests for people working in noncommutative geometry. 
The results are  
\begin{enumerate}
\item[$\bullet$] inverse function theorem (Proposition \ref{anal-inverse}),
\item[$\bullet$] Moser's trick (Proposition \ref{Moser}),
\item[$\bullet$] separation lemma (Proposition \ref{separation}),
\item[$\bullet$] finite determinacy of Jacobi-finite potential (Theorem \ref{finite-determinacy}),
\item[$\bullet$] noncommutative Mather--Yau theorem (Theorem \ref{ncMY}).
\end{enumerate}
Among them, Theorem \ref{finite-determinacy} answers a question in 
Donaldson--Thomas theory posed by Ben Davison in Remark 3.10 of \cite{Dav19}.

\paragraph{Donaldson--Thomas theory of quivers with analytic potentials} As an application of noncommutative differential calculus of analytic potentials, we construct a canonical perverse sheaf on the moduli stack of finite dimensional modules over the Jacobi algebra. 
\begin{theorem}(Theorem \ref{pervsheaf-QP}, Corollary \ref{cor:perv-J-finite}, Proposition \ref{anal-mutation}, Proposition \ref{nondeg-anal}, Corollary \ref{perv-mut}) \label{mainthm}
Let $Q$ be a finite quiver and $\Phi$ be an analytic potential. Denote by $\wh{\Lm}(Q,\Phi)$ and $\wt{\Lm}(Q,\Phi)$ the formal and analytic Jacobi algebra. 
\begin{enumerate}
\item[$(1)$] There is a canonical perverse sheaf of vanishing cycles defined on the moduli stack of finite dimensional modules over $\wh{\Lm}(Q,\Phi)$.
\footnote{
In general $\wh{\Lm}(Q,\Phi)$ and $\wt{\Lm}(Q,\Phi)$ are not necessarily isomorphic. However, they have equivalent categories of finite dimensional modules (see Proposition \ref{anal-formal}). So for the study of DT theory, it is equivalent to consider $\wh{\Lm}(Q,\Phi)$ or $\wt{\Lm}(Q,\Phi)$.
}
\item[$(2)$] If $\wt{\Lm}(Q,\Phi)$ is finite dimensional then the above mentioned perverse sheaf only depends on the isomorphism class of the pair $(\wt{\Lm}(Q,\Phi),[\Phi]_\Phi)$ where $[\Phi]_\Phi$ is the class of $\Phi$ in $\wt{\Lm}(Q,\Phi)_\cy$.
\end{enumerate}
Let $Q$ be a finite quiver with neither loops nor 2-cycles. 
\begin{enumerate}
\item[$(3)$] A mutation of an analytic potential is analytic. 
\item[$(4)$] There exists a nondegenerate analytic potential.
\item[$(5)$] If $\Phi$ is nondegenerate, then there is a canonical perverse sheaf on the moduli stack of finite dimensional modules over the formal Jacobi algebra of arbitrary iterated mutations of $(Q,\Phi)$.
\end{enumerate}
\end{theorem}
When the potential $\Phi$ is algebraic, i.e. contains only finitely many terms, this perverse sheaf was constructed in \cite{KS08} for the {\bf{algebraic}} Jacobi algebra. However, it is not compatible with mutation since the mutation of an algebraic potential is no longer algebraic in general (see Remark 0.1 of \cite{Nagao}). 
Parts $(1), (3)$ and $(5)$ of the above theorem provide a solution to this problem. We believe that the perverse sheaf that we constructed can be used to define refined DT invariants for quivers with analytic potentials, a problem which is left 
for future research.
 
Part $(2)$ leads to a rigidity result for DT theory. Let $\Phi$ be a quasi-homogeneous analytic potential with finite dimensional Jacobi algebra. Then its refined DT theory is completely determined by the isomorphism class of the Jacobi algebra. If  the quasi-homogeneous assumption is dropped, then we get a slightly weaker statement, i.e. the refined DT theory is determined by the pair $(\wt{\Lm}(Q,\Phi),[\Phi]_\Phi)$. There are some interesting examples from geometry and algebra with finite dimensional Jacobi algebras, e.g. the potentials coming from 3-dimensional flopping contractions (see \cite{HuaKeller}, \cite{Dav19}) or any potential on a quiver admitting a
reddening sequence, cf. \cite{Keller11, BruestleDupontPerotin14, Keller19}.

\begin{theorem} (Proposition \ref{integration}, Theorem \ref{thm:Nagao})
Let $Q$ be a finite quiver and $\Phi$ an analytic potential. The weighted integration map $I^\nu$ is a Poisson algebra morphism. If $Q$ has neither loops nor 2-cycles and 
$\Phi$ is nondegenerate, then the Donaldson--Thomas invariants defined by $I^\nu$ satisfy Nagao's transformation formula.
\end{theorem}
Our theorem leads to a new version of the $F$-series, which can be called \emph{the perverse F-series}. Recall that the $F$-series is called \emph{F-polynomial} 
in cluster theory, where it is indeed a polynomial. 
The perverse $F$-series is defined to be the image 
of the quiver Grassmannian under the {\bf{weighted}} integration map. We expect that it can  also be defined as the generating series of the Euler characteristics of the canonical perverse sheaf of vanishing cycles.  We exhibit an example in Section \ref{sec:mutation} showing that the perverse F-series is usually different from the ordinary F-series. As in the geometric case, the perverse $F$-series satisfies the same transformation formula as the ordinary $F$-series under mutation. However, we expect that the perverse $F$-series should behave better under deformation. Extending the work of Nagao \cite{Nagao}
we show that perverse $F$-polynomials yield a Behrend weighted
Caldero--Chapoton formula for what we call anti-cluster algebras.

The paper is organized as follows. After setting up the notation in Section \ref{sec:prelim}, we recall foundations on formal differential calculus of quivers with potentials in Sections \ref{sec:alg} and \ref{sec:formal}. In the rest of Section 3, we establish the foundations of differential calculus of analytic potentials. Two fundamental theorems: the inverse function theorem and Moser's trick (or Cauchy-Kowalevski theorem) for analytic potentials are proved. In Section \ref{sec:ncMY}, we prove the finite determinacy and the Mather-Yau theorem for $\wt{J}$-finite analytic potentials. In Section \ref{sec:DT}, we construct the perverse sheaf of vanishing cycles on the moduli stack of finite dimensional modules over the Jacobi algebra of a quiver with analytic potential. And we prove that the weighted integration map is a Poisson morphism. In Section \ref{sec:mutation}, we prove the separation lemma for analytic potentials and establish the transformation formula for DT invariants under mutation. Finally we define perverse $F$-series and prove the 
Behrend-weighted Caldero--Chapton formula for anti-cluster algebras.

\paragraph{Acknowledgments.} We are very much indebted to Ben Davison for inspiring discussions and for drawing our attention to the work of Toda on Ext-quivers. The Jacobi algebra of an analytic potential has been considered in \cite{Dav19}
in some special case. We are grateful to Yan Soibelman for pointing out that Theorem \ref{integration} would follow from the fact that the quantum integration map defined in \cite{KS08} is an algebra morphism assuming the existence of orientation data and the motivic Milnor fiber identity proved by Le Quy. We thank Fan Qin for several comments on perverse $F$-series.  The first author would like to thank Lev Borisov for useful suggestions for the proof of Proposition \ref{Moser}, and Gui-song Zhou for interesting comments.  The research of the first author is supported by RGC General Research Fund 
no.~17330316,  no.~17308017 and  no.~17308818. 

\section{Preliminaries}\label{sec:prelim}
Throughout, we fix a commutative base ring $k$ and a finite rank
$k$-algebra $l=k e_1+\ldots +k e_n$ for central orthogonal 
primitive idempotents $e_i$. Unadorned tensor products are taken over $k$.

\paragraph{Derivations, double derivation and cyclic derivations.}
An {\em $l$-algebra} $A$ is a $k$-algebra $A$ equipped with a $k$-algebra 
monomorphism $\eta: l\to A$. Note that the image of $l$ is in general not central even though $l$ is commutative. 

We denote $A\ot A^{op}$ by $A^e$.
We write $A\stackrel{out}{\ot}A$ (resp. $A\stackrel{in}{\ot} A$) 
for the $A$-bimodule $A\ot A$ endowed with the outer (resp. inner) 
action of $A^e$.
Because $l$ is a sub algebra of $A$, these $A$-bimodules 
are in particular $l$-bimodules. 
The flip map $\tau:A\stackrel{out}{\ot}A \to A\stackrel{in}{\ot}A$, which
takes $a^\p\ot a^\pp$ to $a^\pp\ot a^\p$,   is an isomorphism of $A$-bimodules, 
$\mu:A\stackrel{out}{\ot}A \to A$ is a homomorphism of $A$-bimodules but in general $\mu: A\stackrel{in}{\ot}A \to A$ is not. Unless otherwise stated, we simply view $A\ot A$ as the bimodule $A\stackrel{out}{\ot}A$.  Also, the category of $A$-bimodules is denoted by $A\Bimod$.

A \emph{(relative) $l$-derivation of $A$ in an  $A$-bimodule $M$} is defined to be a 
$l^e$-linear map $\delta:A\to M$  satisfying the Leibniz rule, that is $\delta(ab) =\delta(a)b+ a\delta(b)$ for all $a,b\in A$.  
It follows that $\delta(l)=0$ and $\delta(A_{ij})\subset M_{ij}$, where $M_{ij}:=e_i M e_j$.
Denote  by $\der_l(A,M)$ the set of all $l$-derivations of $A$ in $M$, which naturally carries a $k$-module structure.  The elements of
\[
\der_l ( A): =\der_l(A,A) ~~~~~ (\text{resp.} \quad \dder_l (A):= \der_l(A,A\ot A) )
\]
are called the  \emph{$l$-derivations of $A$} (resp. \emph{double $l$-derivations of $A$}). For a general double derivation $\delta\in \dder_l A$ and $a\in A$, we shall  write in Sweedler's notation
\begin{align}\label{Sweedler}
\delta(a)= \delta(a)^\p\ot\delta(a)^\pp.
\end{align}
The inner bimodule structure of $A\ot A$ naturally yields a bimodule structure on $\dder_l(A)$.  In contrast, $\der_l (A)$  does not  have canonical left nor  right $A$-module structures in general.  The multiplication map $\mu$ induces  a $k$-linear map $\mu\circ-: \dder_l(A)\to \der_l(A)$ given by $\delta\mapsto \mu\circ \delta$. We refer to $\mu\circ \delta$ as the $l$-derivation corresponding to the double 
$l$-derivation $\delta$.

Let us put on the space of $k$-module endomorphisms $\Hom_k(A,A)$  the $A$-bimodule structure  defined by
\[
a_1fa_2: b\mapsto a_1 f(b) a_2, ~~~~~ f\in \Hom_k(A,A), ~ a_1, a_2, b\in A.
\]
Though  the map $\dder_l(A) \xrightarrow{\mu\circ- } \Hom_k(A,A)$ does not preserve  the bimodule structures,  the map
\[
\mu\circ \tau\circ- : \dder_l (A) \to \Hom_k(A,A)
\]
is clearly a homomorphism of $A$-bimodules. Denote  the image of this map 
by $\cder_l (A)$ and call its elements \emph{cyclic  $l$-derivations of $A$}. 
For a double $l$-derivation $\delta\in \dder_l(A)$, we shall refer to $\mu\circ \tau\circ \delta$ as the cyclic $l$-derivation corresponding to $\delta$. Note that by definition $\cder_l(A)$ is  an $A$-subbimodule of $\Hom_k(A,A)$, and hence is itself an $A$-bimodule.

We collect some trivial properties o (cyclic) derivations in the following lemma.
\begin{lemma}\cite[Lemma 1.1]{HuaZhou}\label{cycprop}
Let  $A$ be an $l$-algebra and fix an element $\Phi\in A_\cy:=A/[A,A]$. Let $\pi:A\to A_\cy$ be the canonical projection and $\phi\in A$ a representative of $\Phi$.
\begin{enumerate}
\item[$(1)$]  $\xi([A,A]) \subseteq [A,A]$ for every $\xi\in \der_l(A)$. Consequently, the assignment $\der_l(A) \ni \xi \mapsto \pi(\xi (\phi))$ only depends on $\Phi$ and defines a $k$-linear map   $\Phi_{\#}: \der_l (A) \to A_\cy $.
\item[$(2)$]  $D([A,A])=0$ for every $D\in \cder_l (A)$. Consequently,  the assignment $\cder_l(A)\ni D \mapsto D(\phi)$ only depends on $\Phi$ and defines an $A$-bimodule morphism $\Phi_*: \cder_l (A)\to A$.
\item[$(3)$] We have the following commutative  diagram:
\begin{align}
\xymatrix{
\dder_l (A) \ar@{->>}[r]^-{\mu\circ \tau\circ-} \ar[d]^{\mu\circ-} &   \cder_l( A)\ar[r]^-{\Phi_*} & A\ar[d]^{\pi} \\
\der_l ( A) \ar[rr]^{\Phi_{\#}} & & A_\cy .
}   \label{derivation}
\end{align}
Consequently, if $\dder_l(A) \xrightarrow{\mu\circ-} \der_l(A)$ is surjective then $\im (\Phi_{\#}) = \im (\pi\circ \Phi_*)$.
\end{enumerate}
\end{lemma}

\paragraph{Quivers.}
A (finite) quiver $Q$ is a tuple $(Q_0,Q_1,s,t)$ where $Q_0, Q_1$ are finite sets and $s,t$ are maps from $Q_1$ to $Q_0$. Elements of $Q_0$ and $Q_1$ are called nodes and arrows respectively. Given $a\in Q_1$, $s(a)$ and $t(a)$ are called the source and target of $a$. The path algebra of $Q$ with  coefficients in $k$ is denoted by $kQ$. Elements of $kQ$ are finite $k$-linear combination of paths, which will be denoted by $\sum a_w w$. Here $w$ is a path or word (of finite length) and $a_w$ is nonzero only for finitely many $w$. Denote by $|w|$ the length of $w$. For each node $i\in Q_0$, we denote by $e_i$ the lazy (i.e. length $0$) path that starts and ends at $i$. We specialize $l$ to $k Q_0:=\bigoplus_{i\in Q_0} ke_i$. The path algebra $kQ$ is an $l$-algebra. Let $\fm$ be the two-sided ideal generated by the arrows. Denote by $\wh{kQ}$ the $\fm$-adic completion of $kQ$, which we call the \emph{complete path algebra}. An element of $\wh{kQ}$ is an infinite linear combination $\sum_w a_w w$. We denote by $\wh{\fm}$ the (closed) ideal of $\wh{kQ}$ generated by all arrows. Clearly, $\fm=\wh{\fm}\cap kQ$ and $\wh{kQ}/\wh{\fm}=k Q_0$.

\paragraph{Power series.} Let $R=k[[x_1,\ldots,x_k]]$ be the power series ring. We denote an element of $R$ by $\sum a_m m$ with $a_m\in k$ where the sum is over all monomials $m$. Given a vector $F(x)$ of power series, we denote by $[m] F(x)$ the coefficient vector of the monomial $m$. We view elements of $\wh{kQ}$ as  noncommutative formal series. 
Given a vector $N(z)$ of such elements, where $z=Q_1$, we denote by $[w] N(z)$ the coefficient vector of the word (path) $w$. For a  noncommutative formal series $\phi$, define $\ord(\phi)$ to be $\min\{|w||a_w\neq 0\}$. For a noncommutative
power series, the order is defined to be the minimal length of a path with non zero coefficient. For a vector $N(z)$, define $\ord(N(z))$ to be the minimal order of its components. 

Given $i\in Q_0$, let $Q^{(ii)}_1:=\{a\in Q_1| s(a)=t(a)=i\}$, i.e. the set of loops based at $i$. Denote by $k[[Q_1^{(ii)}]]$ the commutative power series ring generated by $Q_1^{(ii)}$. Let $\iota^{(i)}: \wh{k Q}\to k[[Q^{(ii)}_1]]$ be the algebra homomorphism that abelianizes the set of loops $a\in Q_1^{(ii)}$  and maps the other arrows to zero. Clearly, $\sum_{i\in Q_0} \iota^{(i)}$ is the abelianization map of $\wh{k Q}$. For $\phi\in \wh{kQ}$, denote by $\phi^\ab$ its abelianization.

Given a formal series $\phi$ (commutative or noncommutative), we denote its multiplicative inverse (suppose it exists) by $\phi^{-1}$ and its composition inverse  (suppose it exists) by $\phi^{\lg -1\rg}$. For a vector $N(z)$, only the composition inverse $N^{\lg -1\rg}(z)$ makes sense.

\paragraph{Trees.} Let $T$ be a finite rooted tree. We denote by $|T|$ the total number of vertices of $T$ and by $l(T)$ the number of leaves. Denote by $\wh{T}$ the rooted tree obtained by deleting all leaves of $T$. The empty tree is denoted by $\emptyset$. The tree with a single vertex is denoted by $\circ$.
Denote by $\BB^P$ the set of rooted plane binary trees, and by $\BB^P_m$ the subset of such trees with $m$ leaves. As a convention, we let $\emptyset, \circ\in \BB^P$. Given $T_1,T_2\in \BB^P$, denote by $B_+(T_1,T_2)$ the rooted binary tree constructed by attaching $T_1$ and $T_2$ (from left to right) to the root. On the other hand, for any $T\in \BB^P$ that is not empty or a singleton, we may write $T=B_+(T_L,T_R)$ where $T_L$ and $T_R$ are the left and the right branches of $T$. Given $T\in \BB^P\setminus \{\emptyset, \circ\}$, we recursively define 
\[
T!:=|T|\cdot T_L! T_R! \; ,
\] 
where $\circ !=1$.

\section{Differential calculus of quivers with potentials}

\subsection{Algebraic case}\label{sec:alg}

Let $Q$ be a finite quiver and let $l:=k Q_0$. 
Let  
\[
kQ_{\cy}:=kQ/[kQ,kQ]
\] be the trace space of $kQ$. An element of $kQ_\cy$ is called a \emph{potential} of $Q$.  

Apply Lemma \ref{cycprop} with $A=kQ$ and $l=kQ_0$. For $\Phi\in kQ_\cy$, we get a commutative diagram. 
\begin{align}
\xymatrix{
\dder_l (kQ) \ar@{->>}[r]^-{\mu\circ \tau\circ-} \ar@{->>}[d]^{\mu\circ-} &   \cder_l(kQ)\ar[r]^-{\Phi_*} & kQ\ar[d]^{\pi} \\
\der_l (kQ) \ar[rr]^{\Phi_{\#}} & & kQ_\cy .
}   \label{derivation-algebraic}
\end{align}
For each $a\in Q_1$ we have a double derivation
\[\frac{\partial~}{\partial a}: kQ \to kQ\ot kQ, ~~~~ Q_1\ni b\mapsto \delta_{a,b}~e_{s(a)}\ot e_{t(a)}.\]
Moreover, every double derivation of $kQ$ has a unique representation of the form
\begin{align} \label{representation-doub}
\sum_{a\in Q_1} \sum_{u,v} A_{u,v}^{(a)}~ u* \frac{\partial~}{\partial a} * v,~~~~~A_{u,v}^{(a)}\in k,
\end{align}
where $u,v$ run over all paths in $Q$, and $*$ denotes the scalar multiplication of the bimodule structure of $\dder_l(kQ)$. For each $a\in Q_1$, let
\[
D_{a}:= \mu \circ \tau \circ \frac{\partial~}{\partial a} \in \cder_l(kQ).
\]
These cyclic derivations were first studied by Rota, Sagan and Stein \cite{RRS}. By (\ref{representation-doub}), every cyclic derivation of $kQ$ has a decomposition (not necessarily unique) of the form
\begin{align}
\sum_{a\in Q_1} \sum_{u,v} A_{u,v}^{(i)}~ u\cdot D_{a}\cdot v,~~~~~A_{u,v}^{(i)}\in k.
\end{align}
Let $\pi: kQ\to k Q_{\cy}$ be the canonical projection.
Given a potential $\Phi\in kQ_{\cy}$, there are two linear maps  
\begin{eqnarray*}
&~\, \Phi_{\#}: \der_l(k Q) \to kQ_{\cy}, & \quad  \xi \mapsto \pi(\xi(\phi))\\
&  \Phi_*: \cder_l(kQ) \to kQ,& \quad  D\mapsto D(\phi),
\end{eqnarray*}
where $\phi$ is any representative of $\Phi$. Clearly $\Phi_\#$ and $\Phi_*$ are  independent of the choice of the representative $\phi$. Note that $\Phi_*$ is a homomorphism of $kQ$-bimodules. So $\im(\Phi_*)$ is a two-sided ideal of $k Q$. Moreover, it is generated by $D_a\Phi$ for all $a\in Q_1$.
\begin{definition}\label{potential-algebraic}
Let $\Phi\in k Q_{\cy}$ be a potential. The \emph{Jacobi ideal} of $\Phi$, denoted by $J(Q,\Phi)$, is defined to be the ideal $\im(\Phi_*)$. The  associative algebra
\[
\Lm(Q, \Phi) := k Q/ J(Q,\Phi),
\]
is called the  \emph{Jacobi algebra}  associated to $(Q,\Phi)$. 
\end{definition}

\subsection{Formal case}\label{sec:formal}

Clearly, $l$-derivations  of  $\wh{k Q}$  are  uniquely determined by their values at all $a\in Q_1$.  However, this is generally not true for $l$-derivations  of  $\wh{k Q}$ with
values in an arbitrary  $\wh{k Q}$-bimodule. In particular, it is false for the $\wh{k Q}$-bimodule $\wh{k Q}\ot\wh{k Q}$ which admits only finite sums.  Thus we need an alternative  to the algebraic double derivations of $\wh{k Q}$  to deal with noncommutative calculus on $\wh{k Q}$.

Let $\wh{k Q}\wh{\ot}\wh{k Q}$ be the vector space over $k$ whose elements are formal series of the form
$\sum\nolimits_{u,v} a_{u,v}~ u\ot v$, where $u,v$ runs over paths in $Q$ and $a_{u,v}\in k$. This is nothing but the adic completion of $k Q\ot k Q$ with respect to the ideal $\fm\ot k Q+ k Q\ot \fm$. It contains $k Q\ot k Q$  as a subspace  under the identification
\[(\sum_{u} a'_u~u) \ot (\sum_{v} a''_v~v) \mapsto \sum_{u,v} a'_ua''_v~ u\ot v.\]
There are two obvious $\wh{k Q}$-bimodule structures on $\wh{k Q} \wh{\ot} \wh{k Q}$, which we call the outer and the inner bimodule structures  defined respectively by 
\[
a(b^\p\ot b^{\p\p})c:=ab^\p\ot b^{\p\p}c \quad \text{and} \quad a*(b^\p\ot b^{\p\p})*c:=b^\p c\ot ab^{\p\p}.
\]
Unless otherwise stated, we  view $\wh{k Q}\wh{\ot} \wh{k Q}$ as a  $\wh{k Q}$-bimodule with respect to the outer bimodule structure.

We write 
\[
\wh{\dder}_l(\wh{k Q}):= \wh{\der}_l(\wh{k Q}, \wh{k Q}\wh{\ot}\wh{k Q})
\]
and call its elements \emph{(formal) double derivations} of $\wh{k Q}$. The inner bimodule structure on $\wh{k Q}\wh{\ot}\wh{k Q}$ naturally yields a bimodule structure on
$\wh{\dder}_l(\wh{k Q})$.
For any  $\delta \in \wh{\dder}_l(\wh{k Q})$ and any $a\in \wh{k Q}$, we write $\delta(a)$ in Sweedler's notation as 
\begin{align}
\label{Sweedler}
\delta(a)= \delta(a)^\p\ot\delta(a)^\pp.
\end{align}
One has to bear in mind  that this notation is an infinite sum. Clearly, double derivations of $\wh{k Q}$ are uniquely determined by their values on arrows. Thus, for each $a\in Q_1$, we have a double derivations
\[\frac{\partial~}{\partial a}: \wh{k Q} \to \wh{k Q}\wh{\ot} \wh{k Q}, ~~~~ Q_1\ni b\mapsto \delta_{a,b}~e_{s(a)}\ot e_{t(a)}.\]
Moreover, every double derivation of $\wh{k Q}$ has a unique representation of the form
\begin{align}\label{representation-doub-formal}
\sum_{a\in Q_1} \sum_{u,v} A_{u,v}^{(a)}~ u* \frac{\partial~}{\partial a} * v,~~~~~A_{u,v}^{(a)}\in k,
\end{align}
where $u,v$ run over all paths in $Q$, and $*$ denotes the scalar multiplication of the bimodule structure of $\dder_l(\wh{k Q})$. The infinite sum (\ref{representation-doub-formal}) makes sense in the obvious way. It is easy to check that $\wh{\dder}_l(\wh{kQ})$ is isomorphic to the $(\fm\ot kQ +kQ\ot \fm)$-adic  completion of $\dder_l(kQ)$ as 
a $\wh{kQ}$-bimodule.


There are two obvious  linear maps $\wh{\mu}: \wh{k Q}\wh{\ot} \wh{k Q}\to \wh{k Q}$ and $\wh{\tau}: \wh{k Q}\wh{\ot} \wh{k Q} \to  \wh{k Q}\wh{\ot} \wh{k Q}$ given respectively by
\[
\wh{\mu}(\sum_{u,v} a_{u,v} u\ot v) =  \sum_{w} (\sum_{w=uv} a_{u,v})~w \quad \text{and} \quad \wh{\tau} (\sum_{u,v} a_{u,v} u\ot v) = \sum_{u,v} a_{v,u} u\ot v,
\] extending $\mu$ and $\tau$.
We put  on $\Hom(\wh{k Q},\wh{k Q})$  the $\wh{k Q}$-bimodule structure  defined by
\[
a_1\cdot f \cdot a_2: b\mapsto a_1 f(b) a_2, ~~~~~ f\in \Hom(\wh{k Q},\wh{k Q}), ~ a_1, a_2, b\in \wh{k Q}.
\]
Then,  although  the map $\wh{\dder}_l(\wh{k Q}) \xrightarrow{\wh{\mu}\circ- } \Hom(\wh{k Q},\wh{k Q})$ does not preserve  the bimodule structures, the map
\[
\wh{\mu}\circ \wh{\tau}\circ- : \wh{\dder}_l (\wh{k Q}) \to \Hom(\wh{k Q},\wh{kQ})
\]
is clearly a homomorphism of $\wh{k Q}$-bimodules. We write
$$
\wh{\cder}_l(\wh{k Q}):= \im(\wh{\mu} \circ \wh{\tau} \circ-)
$$
and call its elements  \emph{(formal) cyclic derivations} of $\wh{k Q}$. 
Note that by definition $\wh{\cder}_l(\wh{k Q})$ is  a $\wh{k Q}$-sub-bimodule of $\Hom(\wh{k Q},\wh{k Q})$, and hence is itself an $\wh{k Q}$-bimodule. For each $a\in Q_1$, let
\[
D_{a}:= \mu \circ \tau \circ \frac{\partial~}{\partial a} \in \wh{\cder}_l(\wh{k Q}).
\]
Every cyclic derivation of $\wh{k Q}$ has a decomposition (not necessarily unique) of the form
\begin{align}\label{representation-cyc}
\sum_{a\in Q_1} \sum_{u,v} A_{u,v}^{(a)}~ u\cdot D_{a}\cdot v,~~~~~A_{u,v}^{(a)}\in k.
\end{align}

Let 
$$\wh{k Q}_{\cy}:= \wh{k Q}/[\wh{k Q},\wh{k Q}]^{cl}$$ with $[\wh{k Q},\wh{k Q}]^{cl}$ being the adic closure of the algebraic commutators.
Elements of  $\wh{kQ}_\cy$ are called \emph{(formal) potentials} of $\wh{k Q}$.  Let $\wh{\pi}: \wh{k Q}\to \wh{k Q}_{\cy}$ be the canonical projection. Given a potential $\Phi\in \wh{k Q}_{\cy}$, there are two linear maps  
\begin{eqnarray*}
&~\, \wh{\Phi}_{\#}: \wh{\der}_l(\wh{k Q}) \to \wh{k Q}_{\cy}, & \quad  \xi \mapsto \wh{\pi}(\xi(\phi))\\
&  \wh{\Phi}_*: \wh{\cder}_l(\wh{k Q}) \to \wh{k Q},& \quad  D\mapsto D(\phi),
\end{eqnarray*}
where $\phi$ is any representative of $\Phi$. Since all derivations and cyclic derivations of $\wh{k Q}$ are continuous,  $\xi([\wh{k Q},\wh{k Q}]^{cl}) \subseteq [\wh{k Q},\wh{k Q}]^{cl}$ for each  derivation $\xi\in \wh{\der}_l(\wh{k Q})$ and $D([\wh{k Q},\wh{k Q}]^{cl})=0$ for each cyclic derivation $D\in \wh{\cder}_l(\wh{k Q})$, whence $\wh{\Phi}_{\#}$ and $\wh{\Phi}_*$ are  independent of the choice of $\phi$. Note that $\wh{\Phi}_*$ is a homomorphism of $\wh{k Q}$-bimodules. So $\im(\wh{\Phi}_*)$ is a two-sided ideal of $\wh{k Q}$. Moreover, it is easy to check that the following diagram is commutative:
\begin{align}
\xymatrix{
\wh{\dder}_l (\wh{k Q}) \ar@{->>}[r]^-{\wh{\mu}\circ \wh{\tau}\circ-}  \ar@{->>}[d]^-{\wh{\mu}\circ-} &   \wh{\cder}_l (\wh{k Q}) \ar[r]^-{\wh{\Phi}_*} & \wh{k Q}\ar@{->>}[d]^-{\wh{\pi}} \\
\der_l (\wh{k Q}) \ar[rr]^-{\wh{\Phi}_{\#}} & & \wh{k Q}_{\cy}.
}  
\end{align}

\begin{definition}\label{potential-complete}
Let $\Phi\in \wh{k Q}_{\cy}$ be a potential. The \emph{(formal) Jacobi ideal} of $\Phi$, denoted by $\wh{J}(Q,\Phi)$, is defined to be the ideal $\im(\wh{\Phi}_*)$. 
The  associative algebra
\[
\wh{\Lm}(Q, \Phi) := \wh{k Q}/ \wh{J}(Q,\Phi),
\]
is called the (formal) \emph{Jacobi algebra}  associated to $\Phi$. 
\end{definition}
If $k$ is a field then $\wh{kQ}$ and $\wh{\dder}_l(\wh{kQ})$ are pseudocompact vector spaces with respect to the adic topology, i.e. the topology is generated by 
the subspaces of finite codimension. Pseudocompact vector spaces form an abelian category (see page 393 of \cite{Gab}). In particular, the image of any continuous (linear) map between pseudocompact vector spaces is closed.

For simplicity, we write $\wh{\Lm}:= \wh{\Lm}(Q,\Phi)$ and $\wh{J}=\wh{J}(Q,\Phi)$.
Let $\wh{\fm}_\Phi\subset \wh{\Lm}$ be the image of $\wh{\fm}$. By the above comment, the Jacobi ideal $\wh{J}\subset \wh{\CC Q}$ is closed. Therefore the $\wh{\fm}_\Phi$-adic topology on $\wh{\Lm}$ is Hausdorff.  We define 
$\wh{\Lm}_{\cy}$ to be the quotient space $\wh{\Lm}/[\wh{\Lm},\wh{\Lm}]^{cl}$ (which coincides with the topological Hochschild homology group $\HoH_0(\wh{\Lm})$. By the closeness  of the range, we have
\[
\wh{\Lm}/[\wh{\Lm},\wh{\Lm}]^{cl}=\wh{kQ}/\Big([\wh{kQ},\wh{kQ}]+\wh{J}\Big)^{cl}=\wh{kQ}/\Big([\wh{kQ},\wh{kQ}]^{cl}+\wh{J}\Big).
\]
The second equality holds because $[\wh{kQ},\wh{kQ}]^{cl}+\wh{J}$ is the image of the continuous addition map  $[\wh{kQ},\wh{kQ}]^{cl}\op\wh{J}\to\wh{kQ}$.
As a consequence, we get a natural map $\wh{kQ}_{\cy}\to \wh{\Lm}_{\cy}$. We denote the image of $\Phi$ by $[\Phi]_\Phi$.
If  $\wh{\Lm}$ is finite dimensional, then the $\wh{\fm}_\Phi$-adic topology on it is discrete and $\wh{\Lm}_{\cy}=\wh{\Lm}/[\wh{\Lm},\wh{\Lm}]$.

\begin{definition}\label{quasi-homogeneous}
We call a potential $\Phi \in \wh{kQ}_{\cy}$ \emph{quasi-homogeneous} if $[\Phi]_\Phi$ is zero in $\wh{\Lm}(Q, \Phi)_{\cy}$, or equivalently,  if  $\Phi$ is contained in $\wh{\pi}(\wh{J}(Q,\Phi))$.
\end{definition}

\begin{definition}
Let $k$ be a field.
 A potential $\Phi\in \wh{kQ}_{\cy}$ is called $\wh{J}$-finite if its Jacobi algebra $\wh{\Lm}(Q,\Phi)$ is finite dimensional. 
\end{definition}

We denote by $\wh{\cG}:=\Aut_l(\wh{kQ},\wh{\fm})$ the group of $l$-algebra automorphisms of $\wh{kQ}$ that preserve $\wh{\fm}$. It is a subgroup of  $\Aut_l(\wh{kQ})$, the group of all $l$-algebra automorphisms of $\wh{kQ}$.  In the case where $k$ is a field, the group $\wh{\cG}$ equals $\Aut_l(\wh{F})$. It acts on $\wh{kQ}_{\cy}$ in the obvious way.

\begin{definition}\label{right-equivalent}
For potentials $\Phi,\Psi\in \wh{kQ}_{\cy}$, we say $\Phi$ is \emph{(formally) right equivalent} to $\Psi$ and write $\Phi \sim \Psi$, if $\Phi$ and $\Psi$ lie in the same $\wh{\cG}$-orbit.
\end{definition}


Given an integer $r\geq0$, the \emph{$r$-th jet space} of $\wh{kQ}$ is defined to be the quotient $l$-algebra $\cJ^r:=\wh{kQ}/\wh{\fm}^{r+1}$. Clearly, the projection map $\wh{kQ}\to \cJ^r$ induces a canonical surjective map
\[q_r: \wh{kQ}_{\cy} \to \cJ^r_\cy:=\cJ^r/[\cJ^r,\cJ^r]\]
with kernel $\pi(\wh{\fm}^{r+1})$.  The image of a potential $\Phi \in \wh{kQ}_{\cy}$ under this map is denoted by $\Phi^{(r)}$.  

Given an integer $r\geq0$, let $\mathcal{G}^r$ be the group of all $l$-algebra automorphisms of $\cJ^r=\wh{kQ}/\wh{\fm}^{r+1}$ preserving $\wh{\fm}/\wh{\fm}^{r+1}$. Clearly, the canonical map $\wh{\cG}\to \mathcal{G}^r$ is surjective. A potential $\Phi\in \wh{kQ}_{\cy}$ is called \emph{$r$-determined} (with respect to $\wh{\cG}$)  if $\Phi^{(r)}\in \cG^r\cdot \Psi^{(r)}$ implies $\Phi \sim \Psi $ for all $\Psi\in \wh{kQ}_{\cy}$.  It is  equivalent to the condition that the equality $\Phi^{(r)}=\Psi^{(r)}$ implies 
$\Phi\sim\Psi$ for all $\Psi \in \wh{kQ}_{\cy}$. 
If $\Phi$ is $r$-determined for some $r\geq 0$, then it is called 
\emph{finitely determined} (with respect to $\wh{\cG}$).

\subsection{Analytic series and its properties}
In this subsection, we take the commutative base ring $k$ to be $\CC$. 

\begin{definition}
Given a positive constant $C$,
a formal series 
\[
\phi=\sum_{w} a_w w
\]
of $\wh{\CC Q}$ is called \emph{analytic of convergence radius $1/C$} if there exists $0<C_1<C$ $|a_w|\leq C_1^{|w|}$ 
for $w$ such that $|w|\gg 0$. Denote by $\wt{\CC Q}_C$ the subspace consisting of such series. We simply call a series \emph{analytic} if there exists some $C>0$ such that $|a_w|\leq C^{|w|}$.  We denote by $\wt{\CC Q}$ the subspace consisting of analytic series.
\end{definition}
It is clear that
\[
\wt{\CC Q}=\bigcup_{C>0} \wt{\CC Q}_C. 
\]
\begin{lemma}\label{lem-coeff}
A formal series $\sum_{w} a_w w$ is    analytic if and only if there exists a positive constant $C$ such that $\sum_{|w|=n} |a_w| \leq C^n$ for any $n\geq 0$. 
\end{lemma}
\begin{proof}
The if part is obvious. Suppose that $|Q_1|=k$. The only if part follows from the inequality
\[
\sum_{|w|=n} |a_w|\leq k^n C^n.
\] 
\end{proof}

\begin{lemma}
Both $\wt{\CC Q}$ and $\wt{\CC Q}_C$ are subalgebras of $\wh{\CC Q}$. 
\end{lemma}
\begin{proof}
We will prove the $\wt{\CC Q}_C$ case. 
It suffices to prove that $\wt{\CC Q}_C$ is closed under multiplication. 
Given two analytic series
\[
\phi:=\sum_{u} a_u u,~~~~\psi:=\sum_v b_v v,
\] 
we have
\[
[w](\phi\cdot\psi)=\sum_{w=uv}a_ub_v. 
\]
Assume that $|a_u|\leq C_1^{|u|}$ and $|b_v|\leq C_1^{|v|}$ for $C_1<C$ and $|u|, |v|\gg 0$.
The claim follows from the inequality
\[
\left|\sum_{w=uv}a_ub_v\right|\leq (|w|+1) C_1^{|w|}\leq (C_1+\ep)^{|w|}
\] for any $\ep>0$ and $|w|\gg 0$.
\end{proof}

\begin{lemma} \label{lemma:exponential-inverse}
Let $\phi$ be an   analytic series without a constant term. 
Then $e^\phi$ is also analytic. If $\phi$ is an analytic series with constant term 1, then $\phi$ is a unit in $\wt{\CC Q}$.
\end{lemma}
\begin{proof}
Let $\phi=\sum_{u, |u|>0} a_u u$. Then the first claim follows from
\begin{align*}
\left|[w]e^\phi\right|&=\left|\sum_{d=1}^{|w|}\frac{1}{d!}\sum_{w=u_1u_2\ldots u_d} a_{u_1}\ldots a_{u_d}\right| \\
&\leq \sum_{d=1}^{|w|}\frac{1}{d!} {{|w|-1} \choose {d-1}} C^{|w|}\\
&\leq (2C)^{|w|} \; .
\end{align*}
Suppose the series
\[
\phi=1-\sum_{|w|\geq 1} a_w w
\] 
satisfies that $|a_w|<C^{|w|}$. Clearly $\phi$ admits a formal inverse. It is indeed analytic by the estimate
\begin{align*}
\left|[w]\phi^{-1}\right|&=\left|\sum_{d=1}^{|w|}\sum_{w=u_1u_2\ldots u_d} a_{u_1}\ldots a_{u_d}\right| \\
&\leq \sum_{d=1}^{|w|} {{|w|-1} \choose {d-1}} C^{|w|}\\
&\leq (2C)^{|w|}
\end{align*}
\end{proof}
It follows from the previous lemma that $\wt{\fm}:=\wh{\fm}\cap \wt{\CC Q}$ is the Jacobson radical of $\wt{\CC Q}$. 

\begin{lemma}\label{lem-ab}
Suppose $\phi\in \wh{\CC Q}$ is analytic. Then each component of $\phi^\ab$ is analytic.
\end{lemma}
\begin{proof}
Recall from Section \ref{sec:prelim} that the components of $\phi^\ab$ are the images of 
the maps $\iota^{(i)}: \wh{\CC Q}\to \CC[[Q_1^{(ii)}]]$.
Without loss of generality, we may assume that $Q$ has a single node and $k$ loops. Suppose that $\phi=\sum_w a_w w$ and $|a_w|\leq C^{|w|}$. Let 
\[
\phi^\ab=\sum_{m} A_m m
\] summed over all monomials. We denote by $|m|$ the degree of the monomial $m$. Then 
\[
\sum_{|m|=n} |A_m|\leq \sum_{|w|=n} |a_w|.
\] By Lemma \ref{lem-coeff}, we have
$\sum_{|m|=n}|A_m|\leq (kC)^{n}$ and $\phi^\ab$ is convergent in the polydisc 
of radius $\frac{1}{kC}$.
\end{proof}
The opposite statement is clearly wrong. For example, the abelianization of
\[
\sum_n n! (xy-yx)^n
\]
 is zero.  

\begin{definition}
An endomorphism $H$ of $\wh{\CC Q}$ is called \emph{analytic} if it maps   analytic series to analytic series.
\end{definition}
\begin{lemma}\label{Ha}
An endomorphism $H$ is   analytic if and only if for any $a\in Q_1$, $H(a)$ is   analytic.
\end{lemma}
\begin{proof}
Set $k=|Q_1|$.
The only if part is obvious. Let 
\[
H(a)=\sum_{u} \gamma_{a,u}u
\] such that $|\gamma_{a,u}|\leq C^{|u|}$ for any $a\in Q_1$.
Given an  analytic series $\phi=\sum_v b_v v$, we have
\[
H(\phi)=\sum_v b_v H(v).
\]
Without loss of generality, we may assume that $b_v\leq C^{|v|}$.
Since  $H$ preserves the ideal $\fm$, for a path $w$, the $w$-coefficient of $H(\phi)$ 
receives only contributions by those $v$ such that $|v|\leq |w|$. Given $a_1,\ldots,a_d\in Q_1$ with $d\leq |w|$, we have
\[
\left|[w]\Big(H(a_1)H(a_2)\ldots H(a_d)\Big)\right|
\leq
\left|\sum_{w=u_1u_2\ldots u_d} \gamma_{a_1,u_1}\ldots\gamma_{a_d,u_d}\right|\leq {{|w|-1} \choose {d-1}} C^{|w|}.
\]
Therefore, the series $H(\phi)$ is analytic since we have
\[
\left|[w]H(\phi)\right|\leq \left|\sum_{v=a_1\ldots a_d, d\leq |w|}b_v\sum_{w=u_1u_2\ldots u_d} \gamma_{a_1,u_1}\ldots\gamma_{a_d,u_d}\right|\leq\sum_{d=1}^{|w|} k^d{{|w|-1} \choose {d-1}} C^{|w|+d}\leq (2kC^2)^{|w|}
\]
\end{proof}

The following proposition is the noncommutative 
version of the analytic  inverse function theorem.
\begin{prop}\label{anal-inverse}
Let $H$ be an analytic endomorphism. If $H$ induces an isomorphism on $\wt{\fm}/\wt{\fm}^2$ then $H^{\lg -1\rg}$ is analytic.
\end{prop}

\begin{proof}
Since $\wt{\fm}/\wt{\fm}^2\cong \wh{\fm}/\wh{\fm}^2$, by \cite[Lemma 2.13]{HuaZhou},
the endomorphism $H$ is formally invertible. 
Without loss of generality, we assume that $H$ induces the identity map on $\wh{\fm}/\wh{\fm}^2$. Let $Q_1=\{z_1,\ldots,z_k\}$ and let $z=(z_1,\ldots,z_k)$.  Set
\[
H(z)=z-M(z)=z- \sum_{|u|>1} A_{u} u \; ,
\] 
where $A_u=(A_{1,u},\ldots,A_{k,u})\in \CC^k$. Let
\[
H^{\lg-1\rg}(z)=z+N(z)=z+\sum_{|w|>1} B_{w} w
\] 
be the formal  inverse of the endomorphism $H$. 

By comparison of coefficients, we find that $B_{w}$ can be expressed as a vector of polynomials $P_{w}=(P_{1,w},\ldots,P_{k,w})$ in $A_{i,u}$ for $i=1,\ldots, k$ with $|u|\leq |w|$. We now recall a combinatorial description of $P_w$ due to Zhao (Theorem 6.2  \cite{Zhao}). Given $z_i\in Q_1$ and a formal series $F_i(z)\in \wh{\CC Q}$, denote by $F_i\delta_{z_i}$ the derivation that sends $z_i$ to $F_i$ and $z_j$ to 0 if $j\neq i$. Given a vector $F(z)=(F_1(z),\ldots,F_k(z))$, denote by $F(z)\delta_z$ the derivation $(F_1\delta_{z_1},\ldots, F_k\delta_{z_k})$. Define a sequence of vectors of formal series $N_T(z)$ indexed by $T\in \BB^P$ as follows:
\begin{align*}
N_T(z)=\begin{cases}
z & T=\emptyset\\
M(z) & T=\circ\\
(N_{T_L}(z)\delta_{z}) N_{T_R}(z) & T=B_+(T_L,T_R).
\end{cases}
\end{align*}
Intuitively,  $N_T(z)$ can be viewed as a rooted binary tree $T$, whose leaves are decorated by $M(z)$, whose left pointed edges are decorated by the operator $1$ 
and whose right pointed edges  are decorated by $\delta_z$.  
At every node, we take the product of the output from the left branch 
with the output from the right branch. Then we have
\[
N(z)=\sum_{T\in \BB^P\setminus \emptyset} \frac{1}{\wh{T}!} N_T(z)=\sum_{m\geq 1} \sum_{T\in \BB^P_m}\frac{1}{\wh{T}!} N_T(z).
\]
It is easy to check that we have $\ord( N_T(z))\geq l(T)+1$. Therefore, for any word $w$,
the following holds
\[
P_w=[w]N(z)=[w]\left( \sum_{m=1}^{|w|-1} \sum_{T\in \BB^P_m}\frac{1}{\wh{T}!} N_T(z)\right).
\]
Recall that $M(z)=\sum_{u,|u|>1} A_u\cdot  u$. For $T\in \BB^P_m$, the components of 
$[w] N_T(z)$ are finite sums of monomials 
$A_{i_1,u_1}A_{i_2,u_2}\ldots A_{i_m,u_m}$ where $|u_1|+\ldots+|u_m|=|w|+m-1$. 
If we set $\deg(A_{i,u})=|u|$, then we have
$\deg(P_{i,w})=2|w|-2$. From the definition of $N_T(z)$, we see that $P_{i,w}$ has positive coefficients. As a consequence, we must have
\[
\Big|P_{i,w}(A_{j,u})\Big|< P_{i,w}(A_{j,u}=1)\cdot C^{2|w|-2}.
\]
Note that $P_{i,w}(A_{j,u}=1)$ is simply $B_{i,w}$ for the special series $M(z)=\sum_{u,|u|>1} u$. Hence it suffices to show that such a special $H(z)=z-M(z)$ admits 
an analytic inverse for composition.
 
Now we assume that $Q_0$ is a singleton. Later we will prove that the general case can be reduced to this special case. Since $Q_0$ is a singleton,
$\wh{\CC Q}$ is the complete free algebra $\CC\lgg z_1,\ldots,z_k\rgg$, whose abelianization is the formal power series ring $\CC[[z_1,\ldots,z_k]]$.
Given an automorphism $H\in \Aut_l(\wh{\CC Q})$, define $H^\ab(z):=(H(z))^\ab$, 
which is an automorphism of $\CC[[z_1,\ldots,z_k]]$.
Denote by $M(z)$ the formal series $\sum_{u,|u|>1} u$. Then $M^\ab(z)$ is analytic, and $H^\ab(z)=z-M^\ab(z)$ is an analytic automorphism of $\CC[[z]]$.
Since $\Big(H^{\lg-1\rg}\Big)^\ab=\Big(H^\ab\Big)^{\lg-1\rg}$, we have
\[
\Big(H^\ab(z)\Big)^{\lg-1\rg}=z+ N^\ab(z). 
\]
By the inverse function theorem for analytic maps, there exists $C>0$ such that 
we have
\[
\sum_{m,|m|=d}\Big|[m]N^\ab(z)\Big|\leq C^d
\]
for any $d\geq 2$.

Since the coefficients $[w]N(z)$ are positive for any word $w$, we have
\[
\sum_{|w|=d} [w]N(z)=\sum_{|m|=d} [m]N^\ab(z)\leq C^d,
\]  
which means that $N(z)$ is analytic.

The above equality does not hold for $\wh{\CC Q}$ for a general quiver because the abelianization only remember loops. For a general quiver $Q$ with $|Q_1|=k$, consider the $k$-loop quiver $q_k$ and fix a bijection between $Q_1$ and set of loops of $q_k$. For $H\in \Aut_l(\wh{\CC Q})$, we consider the automorphism $h\in \Aut(\wh{\CC q_k})$
obtained by replacing arrows in $Q$ with the corresponding loops in $q_k$ 
in the expression for $H$. Without loss of generality, we set $A_{i,u}=1$ for all $z_i\in Q_1$ and $u$. Set
\[
B^Q_{w}=[w] H^{\lg -1\rg}(z),~~~ B^q_{w}=[w] h^{\lg -1\rg}.
\]
Since we have
\[
B_{w}^Q= B_{w}^q,
\] 
the general case follows from the $k$-loop case. 
\end{proof}

\subsection{Analytic derivations and analytic Jacobi algebra}

As a consequence of the inverse function theorem, the group of formal automorphisms
$\Aut_l(\wh{\CC Q})$ has a well-defined subgroup
$\Aut_l(\wt{\CC Q})$ consisting of the analytic $l$-automorphisms. 
Since we work over $\CC$, all $l$-automorphisms  preserve $\wt{\fm}$.  

The set of $l$-derivation of $\wt{\CC Q}$, denoted by  $\der_l(\wt{\CC Q})$, forms a Lie subalgebra of $\der_l(\wh{\CC Q})$. Denote by $\der_l^+(\wt{\CC Q})$ its Lie subalgebra consisting of the $l$-derivations that preserve $\wt{\fm}$.  
Using a similar argument as in the proof of Lemma \ref{Ha}, 
we can show that
a derivation $\delta$ is analytic if and only if $\delta(a)$ is analytic for all $a\in Q_1$. 

Denote by $\wt{\CC Q}\tot\wt{\CC Q}$ the subspace of $\wh{\CC Q}\wh{\ot} \wh{\CC Q}$ consisting of the infinite sums
\[
\sum_{u,v} A_{u,v} u\ot v
\]
such that there exists $C>0$ such that $|A_{u,v}|\leq C^{|u|+|v|}$. 


Clearly, $\wt{\CC Q}\tot\wt{\CC Q}$ 
carries an outer and inner $\wt{\CC Q}$-bimodule 
structure under the obvious actions. Denote by $\wt{\dder}_l(\wt{\CC Q})$ the space of derivations $\der_l(\wt{\CC Q},\wt{\CC Q}\tot\wt{\CC Q})$. It is equipped with a bimodule structure by the inner bimodule structure on $\wt{\CC Q}\tot\wt{\CC Q}$. We denote by $\wt{\dder}^+_l(\wt{\CC Q})$ the subspace consisting of the double derivations that map 
$\wt{\fm}$ to $\wt{\fm}\wt{\ot} \wt{\CC Q} + \wt{\CC Q}\wt{\ot} \wt{\fm}$. 

Let $$\wt{\CC Q}_\cy: = \wt{\CC Q}/ [\wt{\CC Q},\wt{\CC Q}]^{cl}.$$  Note that the map $\wt{\CC Q}_\cy \to \wh{\CC Q}_{\cy}$ induced by the  inclusion $\wt{\CC Q} \to \wh{\CC Q}$ is injective, since we have
\[
[\wt{\CC Q}, \wt{\CC Q}]^{cl} = \bigcap_{n\geq0}\big( [\wt{\CC Q}, \wt{\CC Q}] +\wt{\fm}^n \big) = \bigcap_{n\geq0}\Big( ( [\wh{\CC Q}, \wh{\CC Q}] +\wh{\fm}^n) \cap \wt{\CC Q} \Big)= [\wh{\CC Q},\wh{\CC Q}]^{cl} \cap \wt{\CC Q}.
\]
Thus we may identify $\wt{\CC Q}_\cy$ with a subspace of $\wh{\CC Q}_\cy$ via the canonical injection. Fix $C>0$, we define
\[
\wt{\CC Q}_{\cy,C}: = \wt{\CC Q}_C/ [\wh{\CC Q},\wh{\CC Q}]^{cl}\cap \wt{\CC Q}_C.
\]
We have natural inclusions $\wt{\CC Q}_{\cy,C}\subset \wt{\CC Q}_\cy\subset \wh{\CC Q}_\cy$.

A (formal) potential $\Phi\in \wh{\CC Q}_\cy$ is called \emph{analytic} if it is contained in 
$\wt{\CC Q}_\cy$, or equivalently if it is the image of an analytic series under the map 
$\wh{\pi}: \wh{\CC Q}\to \wh{\CC Q}_\cy$. A formal potential $\Phi$ can be expressed as an infinite sum $\sum_c a_c c$ where the sum is over all cyclic words and the $a_c$ are uniquely determined. In other words, $\wh{\CC Q}$ admits a topological basis given by the cyclic words with the adic topology.
A formal potential $\Phi$ is called \emph{analytic of convergence radius $1/C$} if $\Phi\in \wt{\CC Q}_{\cy, C}$.

For any analytic potential $\Phi$, we have a commutative diagram analogous to the formal case:
\begin{align}
\xymatrix{
\wt{\dder}_l (\wt{\CC Q}) \ar@{->>}[r]^-{\wt{\mu}\circ \wt{\tau}\circ-}  \ar@{->>}[d]^-{\wt{\mu}\circ-} &   \wt{\cder}_l (\wt{\CC Q}) \ar[r]^-{\wt{\Phi}_*} & \wt{\CC Q}\ar[d]^{\wt{\pi}} \\
\der_l (\wt{\CC Q}) \ar[rr]^-{\wt{\Phi}_{\#}} & & \wt{\CC Q}_{\cy},
}   \label{derivation+}
\end{align}
\begin{definition}
The \emph{analytic Jacobi ideal} of $\Phi$, denoted by $\wt{J}(Q,\Phi)$, is defined to be $\im(\wt{\Phi}_*)$. The \emph{analytic Jacobi algebra} is defined to be
\[
\wt{\Lm}(Q,\Phi):=\wt{\CC Q}/\wt{J}(Q,\Phi).
\]
\end{definition}
Since $\wt{J}(Q,\Phi)\subset\wh{J}(Q,\Phi)\cap \wt{\CC Q}$, the analytic and formal Jacobi algebras are related by the following sequence of morphisms:
\[
\wt{\CC Q}/\wt{J}(Q,\Phi)\to \wt{\CC Q}/\wh{J}(Q,\Phi)\cap \wt{\CC Q}\cong \wt{\CC Q}+\wh{J}(Q,\Phi)/\wh{J}(Q,\Phi)\to\wh{\CC Q}/\wh{J}(Q,\Phi),
\]
and the diagram
\[
\xymatrix{\wt{\CC Q} \ar[d] \ar[r] & \wh{\CC Q}\ar[d]\\
\wt{\Lm}(Q,\Phi)\ar[r]\ar[r] & \wh{\Lm}(Q,\Phi)
}
\] commutes.

Define $\wt{\Lm}(Q,\Phi)_\cy$ to be the quotient of $\wt{\Lm}(Q,\Phi)$ by the preimage of $[\wh{\Lm}(Q,\Phi),\wh{\Lm}(Q,\Phi)]^{cl}$ under the canonical map $\wt{\Lm}(Q,\Phi)\to \wh{\Lm}(Q,\Phi)$. As vector spaces, we have
\[
\wt{\Lm}(Q,\Phi)_\cy=\wt{\CC Q}/ \Big([\wt{\CC Q},\wt{\CC Q}]^{cl}+\wt{J}(Q,\Phi)\Big).
\]
By an abuse of notations, we write $[\Phi]_\Phi$ for the class in $\wt{\Lm}(Q,\Phi)_\cy$ represented by $\Phi$. The natural morphism $\wt{\Lm}_\cy\to\wh{\Lm}_{\cy}$ sends $[\Phi]_\Phi$ to the corresponding class of its underlying formal potential in $\wh{\Lm}_{\cy}$.

\begin{definition}\label{quasi-homogeneous-anal}
We call a potential $\Phi \in \wt{\CC Q}_{\cy}$ \emph{quasi-homogeneous} if $[\Phi]_\Phi$ is zero in $\wt{\Lm}(Q, \Phi)_{\cy}$, or equivalently,  if  $\Phi$ is contained in $\wt{\pi}(\wt{J}(Q,\Phi))$.
\end{definition}
If an analytic potential $\Phi$ is quasi-homogeneous, 
then its underlying formal potential is also quasi-homogeneous.

\begin{definition}
 A potential $\Phi\in \wt{\CC Q}_{\cy}$ is called $\wt{J}$-finite if its Jacobi algebra $\wt{\Lm}(Q,\Phi)$ is finite dimensional. 
\end{definition}

Since any analytic potential has an underlying formal potential, we have two notions of $J$-finiteness depending on whether the formal or analytic Jacobi algebra is used. The following lemma clarifies the relation between the Jacobi algebra of an analytic potential and its underlying formal potential.

\begin{lemma}\label{lem-analfd}
Let $\Phi\in \wt{\CC Q}_{\cy}$ be an analytic potential. If $\wt{\Lm}(Q,\Phi)$ is finite dimensional then there is an isomorphism $\wh{\Lm}(Q,\Phi)\cong\wt{\Lm}(Q,\Phi)$.
\end{lemma}
\begin{proof}
For simplicity, we write $\wt{J}=\wt{J}(Q,\Phi)$, $\wh{J}=\wh{J}(Q,\Phi)$, $\wt{\Lm}=\wt{\Lm}(Q,\Phi)$ and $\wh{\Lm}=\wh{\Lm}(Q,\Phi)$.
Since $\wt{\fm}^n=\wh{\fm}^n\cap \wt{\CC Q}$, there is an isomorphism $\wt{\CC Q}/\wt{\fm}^n\cong \wh{\CC Q}/\wh{\fm}^n$. Therefore, $\wh{\CC Q}$ is the $\wt{\fm}$-adic completion of $\wt{\CC Q}$. Similarly $\wh{\Lm}$ is the $\wt{\fm}$-adic completion of $\wt{\Lm}$. If $\wt{\Lm}$ is finite dimensional, then $\wt{J}$ is closed in $\wt{\CC Q}$, i.e. $\cap_n (\wt{\fm}^n+\wt{J})=\wt{J}$. It follows that the canonical map $\wt{\Lm}\to \wh{\Lm}$ is injective. It is surjective because $\wh{\Lm}$ is the completion of $\wt{\Lm}$ and $\wt{\fm}^n\subset \wt{J}$ for $n\gg 0$.
\end{proof}
\begin{remark}
We do not know whether finite dimensionality of $\wh{\Lm}$ will imply finite dimensionality of $\wt{\Lm}$ or not.
\end{remark}
The automorphism group $\wt{\cG}:=\Aut_l(\wt{\CC Q})$ acts on $\wt{\CC Q}_\cy$ naturally.
\begin{definition}\label{an-right-equivalent}
For potentials $\Phi,\Psi\in \wt{\CC Q}_{\cy}$, we say $\Phi$ is \emph{analytically right equivalent} to $\Psi$ and write $\Phi \sim_a \Psi$, if $\Phi$ and $\Psi$ lie in the same $\wt{\cG}$-orbit.
\end{definition}
Since $\wt{\CC Q}/\wt{\fm}^{r+1}\cong \wh{\CC Q}/\wh{\fm}^{r+1}$, we use $\cJ^r$ to denote both quotients.  By the  inverse function theorem \ref{anal-inverse},  the canonical map $\wt{\cG}\to \mathcal{G}^r$ is surjective. The notions of $r$-determined and finitely determined analytic potential are defined in analogy with the
formal case by replacing $\wh{\cG}$ with $\wt{\cG}$. If an analytic potential $\Phi$ is finitely determined, then its underlying formal potential is finitely determined.

\subsection{Moser's trick}
In this section, we prove the (local) integrability of a one parameter family of analytic derivations for quivers with analytic potentials. This can be viewed as a noncommutative geometric version of Moser's trick, or the noncommutative version of the
Cauchy-Kowalevski theorem.  The formal version has been proved in \cite{HuaZhou}. Throughout this section, we let $Q$ be a quiver with $Q_1=\{z_1,\ldots,z_k\}$.

A $U$-family of formal noncommutative series is a formal series 
\[
f(z,t)=\sum_w \lambda_w(t) w
\] 
where the $\lambda_w(t)$ are analytic functions on an open subset 
$U\subset \CC$ for any path $w$. 
Such a family $f(z,t)$ is said to be \emph{analytic} if there exists a positive constant $C$  such that for any $t\in U$
\[
|\lambda_w(t)|\leq C^{|w|}.
\]
 
Let $\bF(z,t):=(F_a(t))_{a\in Q_1}$ be a vector of $U$-families 
of analytic noncommutative series
\[
F_a(t)=\sum_{w, |w|>0} \lambda_{a,w}(t) w.
\]  

Let $\bu(z,t)=(u_b(z,t))_{b\in Q_1}$ be a vector of noncommutative formal series with coefficients being local analytic functions at $0\in\CC$, i.e. 
\[ u_b(z,t)=\sum_w \gamma_{b,w}(t) w.\] 
Define $\partial_t \bu(z,t)$ to be the vector of formal series with coefficients being the $t$-derivatives of those of $\bu(z,t)$. If there exists a neighborhood $U^\p$ of $0$ such that all $\gamma_{b,w}(t)$ are analytic on $U^\p$ then
the vector $\bu(z,t)$ is a $U^\p$-family  of endomorphism of $\wh{\CC Q}$, while $\bF(z,t)$ can be understood as a $U$-family of analytic derivations. The composition 
$\bF(\bu(z,t),t)$ can be interpreted as the pull back of the derivation by the endomorphism. The initial value problem
\begin{equation}
\partial_t\bu(z,t)=\bF(\bu(z,t),t),~~~~\bu(z,0)=z
\end{equation}
is an infinite hierarchy of ordinary differential equations for all $b\in Q_1$
\begin{equation}\label{k=k}
\begin{cases}
\gamma^\p_{b,w}(t)=\sum_{w=u_1u_2\ldots u_d}\sum_{a_1,\ldots,a_d\in Q_1} \lambda_{b,a_1\ldots a_d}(t)\Big( \gamma_{a_1,u_1}(t)\gamma_{a_2,u_2}(t)\ldots \gamma_{a_d,u_d}(t)\Big)\\
\gamma_{b,w}(t)=0 ~~\text{for any}~~w~~\text{such that} ~~ s(b)\neq s(w)~~\text{or}~~ t(b)\neq t(w),\\
\gamma_{b,a}(0)=\delta_{b,a},\\
\gamma_{b,w}(0)=0~~\text{for all $w$ such that}~~|w|\geq 2.
\end{cases}
\end{equation} 
A partition of the word $w=u_1u_2\ldots u_d$ is equivalent with a composition $|w|=n_1+n_2+\ldots +n_d$ (with $n_i>0$). Therefore the sum on the right hand side of the differential equation has $\sum_{d=1}^{|w|} {|w|-1\choose d-1} k^d$ terms.

\begin{prop}\label{Moser}
There exists a neighborhood $V$ of $0$ such that 
system \ref{k=k} has a unique analytic solution $\gamma_{b,w}(t)$ on $V$ for every $b\in Q_1$ and every word $w$ and  
$\bu(z,t)=(u_b(z,t))_{b\in Q_1}$ is invertible.
Moreover,  $\bu(z,t)$ is a vector of $V$-family of analytic series.
\end{prop}
\begin{proof}
The first part of the proposition i.e. the solvability of system \ref{k=k} has already been proved in Lemma 3.17 of \cite{HuaZhou}. 
\footnote{
In \cite{HuaZhou}, the $\lambda_{b,w}(t)$ are assumed to be entire functions. However, the same proof is valid for the case when the defining domain is an arbitrary open subset of complex plane.
}
The solution is a family of formal automorphisms of $\wh{\CC Q}$. 
Here we will prove that by restricting to probably smaller open set we get a family of analytic automorphisms.

We adopt the following probably nonstandard notation. For an analytic function $f(t)$, denote by $f^{(n)}(t)$ the normalized derivative $\frac{1}{n!} \frac{d^nf}{dt^n}(t)$. Let $\delta$ be a positive number such that the closed disc centered at $0$ of radius $\delta$ is contained in $U$. Since $|\lambda_{b,w}(t)|< C^{|w|}$ on $U$, 
by the Cauchy integration formula, we have
\begin{equation}\label{lambda-bound}
|\lambda_{b,w}^{(m)}(0)|\leq \delta^{-m} C^{|w|}.
\end{equation} Define $C_1:=\max\{C, \delta^{-1}\}$. Then $|\lambda_{b,w}^{(m)}(0)|\leq C_1^{|w|+m}$.

When $w=c \in Q_1$, equation \ref{k=k} reduces to
\[\gamma^\p_{b,c}(t)=\sum_{a\in Q_1}\lambda_{b,a}(t) \gamma_{a,c}(t).
\]

If we differentiate with respect to $t$, we find
\begin{align*}
\gamma^{(n+1)}_{b,c}(t)&=\frac{1}{n+1}\sum_{a\in Q_1} \sum_{m_0=0}^n \lambda_{b,a}^{(m_0)}\gamma_{a,c}^{(n-m_0)}\\
&=\frac{1}{n+1}\sum_{a_1,a_2} \sum_{m_1+m_2=n-m_0-1}\frac{1}{n-m_0} \lambda_{b,a_1}^{(m_0)}\lambda_{a_1,a_2}^{(m_1)}\gamma_{a_2,c}^{(m_2)}\\
&=\frac{1}{n+1}\sum_{a_1,a_2,a_3} \sum_{m_2+m_3=n-m_0-m_1-2}\frac{1}{(n-m_0)(n-m_0-m_1-1)} \lambda_{b,a_1}^{(m_0)}\lambda_{a_1,a_2}^{(m_1)}\lambda_{a_2,a_3}^{(m_2)}\gamma_{a_3,c}^{(m_3)}\\ \ldots
\end{align*}
Let us evaluate at $t=0$ and apply the initial condition that $\gamma_{a,c}(0)=\delta_{a,c}$. We obtain
\begin{align}\label{|w|=1}
\gamma_{b,c}^{(n+1)}(0)=\frac{1}{n+1}\sum_{d=0}^{n}\sum_{a_1,\ldots,a_d}\sum_{\sum_{i=0}^d m_i=n-d}\frac{1}{\prod_{i=0}^{d-1}(n-\sum_{j=0}^i m_j - i)}
\lambda_{b,a_1}^{(m_0)}(0)\lambda_{a_1,a_2}^{(m_1)}(0)\ldots\lambda_{a_d,c}^{(m_d)}(0)
\end{align}
The $d=0$ term is $\lambda_{b,c}^{(n)}(0)$.

For a general $w$, in equation \ref{k=k},
we differentiate both sides $n$ times to get
\begin{align}\label{recursion}
\gamma_{b,w}^{(n+1)}(t)&=\frac{1}{n+1}\sum_{w=u_1\ldots u_d}\sum_{a_1,\ldots,a_d\in Q_1}\sum_{m_0+\ldots+m_d=n}\lambda_{b,a_1\ldots a_d}^{(m_0)}(t)\gamma^{(m_1)}_{a_1,u_1}(t)\ldots \gamma^{(m_d)}_{a_d,u_d}(t)
\end{align}
Note that the $u_i$ are subwords of $w$ and $m_i\leq n$. By applying
successively formula \ref{recursion} iteratively, 
formula \ref{|w|=1} and the boundary condition 
\[
\gamma_{a,u}(0)=
\begin{cases}
0 & |u|>1\\
\delta_{a,u} & |u|=1
\end{cases}
\]
we can express $\gamma^{(n+1)}_{b,w}(0)$ as a finite sum of monomials in the $\lambda^{(m)}_{a,u}(0)$, where $a\in Q_1$, $|u|\leq |w|$ and $m\leq n$. Note that the same monomial might occur in the result more than once. However we do NOT combine them. We denote this polynomial by $P^{n+1}_{b,w}$. Clearly, all its coefficients lie in the interval $(0,1]$. Define 
\[
\deg(\gamma^{(n+1)}_{b,w}(0))=n+|w|,~~~~~ \deg(\lambda^{(m)}_{a,u}(0))=m+|u|.
\]
Then equation \ref{recursion} is homogeneous with respect to this grading. It follows that every monomial in $P_{b,w}^{(n+1)}$ has total degree $n+|w|$. By \ref{lambda-bound}, the absolute value of the monomial is bounded by $C_1^{n+|w|}$.

Now we estimate the total number of monomials in $P^{(n+1)}_{b,w}$.
Denote by $N_{a,n,w}$ the number of monomials (again we do NOT simplify!)  in $P^{(n+1)}_{a,w}$. We would like to understand the growth
of $N_{a,n,w}$ with respect to $n$ and $w$. Since $N_{a,n,w}$ only depends on $a, n, |w|$, we write $N_{a,n,|w|}=N_{a,n,w}$.
It follows from the boundary condition that for any $a\in Q_1$,
we have $N_{a,0,1}\leq 1$ and  $N_{a,0,i}=0$ when $i>1$. 
Set $N_{n,|w|}:=\max \{ N_{a,n,|w|}| a\in Q_1\}$.

The recursion \ref{recursion} gives the inequality 
\[
N_{n+1,|w|}\leq \sum_{|u_1|+\ldots+|u_d|=|w|} \sum_{m_0+\ldots+m_d=n} k^d\cdot N_{m_1,|u_1|}\cdot\ldots\cdot N_{m_d,|u_d|}.
\]
Consider the recursive system
\begin{align*}
\begin{cases}
N^\p_{n+1,|w|}= \sum_{|u_1|+\ldots+|u_d|=|w|} \sum_{m_0+\ldots+m_d=n} k^d\cdot N^\p_{m_1,|u_1|}\cdot\ldots\cdot N^\p_{m_d,|u_d|}\\
N^\p_{0,1}=1,~~~~N^\p_{0,i}=0 ~~~~\text{for $i>1$}
\end{cases}
\end{align*}

Since $N_{n,|w|}\leq N^\p_{n,|w|}$, it suffices to bound the growth of $N^\p_{n,|w|}$. We introduce a generating function
\[
G(x,y):=\sum_{n\geq 0, |w|\geq 1} N^\p_{n,|w|} x^ny^{|w|}.
\]
Fix $d\geq 1$ and $m_0\geq 0$. The coefficient of $x^ny^{|w|}$ for $(kG)^d x^{m_0}$ is
\[\sum_{|u_1|+\ldots+|u_d|=|w|} \sum_{m_0+\ldots+m_d=n} k^d\cdot N^\p_{m_1,|u_1|}\cdot\ldots\cdot N^\p_{m_d,|u_d|}.
\]
It follows that the coefficient of $x^ny^{|w|}$ in $\sum_{d\geq 1, m_0\geq 0} (kG)^d x^{m_0}$ is equal to $N^\p_{n+1,|w|}$. This leads to a quadratic  equation satisfied by the generating series:
\[
\frac{G-y}{x}=\sum_{n\geq 0,|w|\geq 1} N^\p_{n+1,|w|} x^ny^{|w|}=\frac{kG}{kG-1}\cdot \frac{1}{1-x}
\]
To show the first equality, we need to apply the initial condition
\[
N^\p_{0,1}=1,~~~~N^\p_{0,i}=0 ~~~~\text{for $i>1$}.
\]
Since the equation can be solved near $x=0, y=0$ analytically, we conclude that 
\[
N_{n,|w|}\leq N^\p_{n,|w|}\leq C_2^{n+|w|}
\] for some constant $C_2$ independent of $n,w$. Therefore, we get
\[
|\gamma^{(n+1)}_{b,w}(0)|\leq |P^{n+1}_{b,w}|\leq C_2^{n+1+|w|}C_1^{n+|w|}\leq K^{n+1+|w|}
\] with $K=\max\{C_1\cdot C_2,1\}$.

By holomorphicity of $\gamma_{b,w}(t)$ at $t=0$ and the initial condition $\gamma_{b,w}(0)=0$, we have
\[
|\gamma_{b,w}(t)|\leq \sum_{n\geq 1} K^{n+|w|} |t|^n\leq K^{|w|}\frac{K|t|}{1-K|t|}\leq K^{|w|}
\] when $|t|\leq \frac{1}{2K}$. Set $V$ to be the disc centered at $0$ of radius $\frac{1}{2K}$.
\end{proof}

\section{Noncommutative Mather-Yau theorem for analytic potentials}\label{sec:ncMY}

Let us begin by fixing some notations. Let $U$ be a proper open subset of the complex plane. We denote  by $K_U$  the $\CC$-algebra of analytic functions on $U$. The base rings $k$ that we need below are $\CC$ and $K_U$. Denote by $\wt{K_UQ}$ the subalgebra of $\wh{K_UQ}$ consisting of the 
elements $\phi_t=\sum_wa_w(t) w$ such that 
\[
|a_w(t)|\leq C^{|w|}
\] in $U$ for some $C>0$ .
For any $s\in U$, there is an evaluation map from $\wh{K_U Q}$ to $\wh{\CC Q}$ 
given by $\phi\mapsto \phi_{s}$. 
Clearly, it restricts to a map from $\wt{K_U Q}$ to $\wt{\CC Q}$.

{\bf{For simplicity of notation, we will omit the subscript $U$ when the open set is fixed.}}
 Let $L= K Q_0$ and $\wt{\fn}$ (resp. $\wh{\fn}$) be the ideal of $\wt{KQ}$ (resp. 
 $\wh{KQ}$) generated by arrows. Let $\wt{KQ}_\cy$ (resp. $\wh{KQ}_\cy$) be the quotient $\wt{KQ}/[\wt{KQ}, \wt{KQ}]^{cl}$ (resp. $\wh{KQ}/[\wh{KQ}, \wh{KQ}]^{cl}$) where $[\wt{KQ}, \wt{KQ}]^{cl}$ (resp. $[\wh{KQ}, \wh{KQ}]^{cl}$) is the $\wt{\fn}$-adic (resp. $\wh{\fn}$-adic)  closure of $[\wt{KQ}, \wt{KQ}]$ (resp. $[\wh{KQ}, \wh{KQ}]$) in $\wt{KQ}$ (resp. $\wh{KQ}$).

Let $\Aut_L(\wh{KQ},\wh{\fn})$ be the group of $L$-automorphisms of $\wh{KQ}$ that preserves $\wh{\fn}$.
Define $\Aut_L(\wt{KQ},\wt{\fn})$ to be the set consisting of the $H\in\Aut_L(\wh{KQ},\wh{\fn})$ such that $H$ preserves $\wt{KQ}$. By the same argument as in the proof of Lemma \ref{Ha} and Proposition \ref{anal-inverse},
$\Aut_L(\wt{KQ},\wt{\fn})$ is a subgroup of $\Aut_L(\wh{KQ},\wh{\fn})$.

We identify $\wt{\CC Q}$ (resp. $\wt{\CC Q}_{\cy}$) with a subspace of $\wt{KQ}$ (resp. $\wt{KQ}_{\cy}$) in the natural way.  Since $l$-algebra automorphisms of $\wt{\CC Q}$ and $L$-algebra automorphisms of  $\wt{KQ}$ are both uniquely determined by their values on the arrows, one may naturally identify the group $\Aut_l(\wt{\CC Q}) = \Aut_l(\wt{\CC Q}, \wt{\fm})$  with a subgroup of  $\Aut_L(\wt{KQ},\wt{\fn})$.

Given $H\in \Aut_L(\wt{KQ},\wt{\fn})$ and $s\in U$, there is an evaluation map 
given by sending $H$ to  $H_s\in \Aut_l(\wt{\CC Q},\wt{\fm})$.
Denote by $\wt{\dder}_L (\wt{K Q})$ the subspace of $\wh{\dder}_L(\wh{KQ})$ consisting of the double derivations 
\[
\delta_t=\sum_{a\in Q_1}\sum_{u,v} A^{(a)}_{u,v}(t) \, u*\frac{\partial}{\partial a} *v
\] 
such that for $t\in U$, we have  $|A^{(a)}_{u,v}(t)|\leq C^{|u|+|v|}$ for some $C>0$. In particular, for any $s\in U$ the specialization $\delta_{s}$ belongs to $\wt{\dder}_l(\wt{\CC Q})$. Similarly, we define $\wt{\cder}_L (\wt{K Q})$ to be $\im(\wt{\mu}\circ\wt{\tau})$.   Given $\Phi\in\wt{KQ}_\cy$, there is a commutative diagram
\begin{align}
\xymatrix{
\wt{\dder}^+_L (\wt{K Q}) \ar@{->>}[r]^-{\wt{\mu}\circ \wt{\tau}\circ-}  \ar@{->>}[d]^-{\wt{\mu}\circ-} &   \wt{\cder}^+_L (\wt{K Q}) \ar[r]^-{\wt{\Phi}_*} & \wt{K Q}\ar@{->>}[d]^-{\wt{\pi}} \\
\der^+_L (\wt{K Q}) \ar[rr]^{{\wt{\Phi}}_{\#}} & & \wt{K Q}_{\cy}.
}  
\end{align}

Given a formal series  $f=\sum_{w} a_w(t) w\in \wh{KQ}$, the derivative  $\frac{df}{dt}$ is defined to be  the formal series $\sum_{w} a^\p_w(t) w$. It is easy to check that taking derivatives preserves the cyclic equivalence relation on $\wh{KQ}$. Consequently, one may naturally define $\frac{d\Phi}{dt}$ for any potential $\Phi\in \wh{KQ}_{\cy}$.

The following three lemmas are the analytic analogies of Lemma 3.19, Lemma 3.3 and Proposition 3.13 of \cite{HuaZhou}. Since the proofs are the same as those in \cite{HuaZhou}, we will skip them. 
\begin{lemma}\label{tangent-compare-1}
For any potential $\Phi\in \wt{\CC Q}_{\cy} \subseteq \wt{KQ}_{\cy} $, we have
\[
\wt{\Phi}_* \big (\wt{\cder}_l^+ (\wt{\CC Q}) \big ) \supseteq \wt{\fm}^r \Longleftrightarrow \wt{\Phi}_*  \big (\wt{\cder}_{L}^+ (\wt{KQ}) \big ) \supseteq \wt{\fn}^r, ~~~ r>0.
\]
\end{lemma}

\begin{lemma}
\label{Jacobi-transform}
Let  $\Phi \in \wt{\CC Q}_{\cy}$  and $H\in \wt{\mathcal{G}}$.  Then
\[
H(\wt{J}(Q,\Phi)) =\wt{J}(Q,  H(\Phi)).
\]
Consequently, $H$ induces an isomorphism of $l$-algebras $\wt{\Lm}(Q, \Phi) \cong \wt{\Lm}(Q,  H(\Phi))$.
\end{lemma}

\begin{prop}\label{Bootstrapping}
Let and $\Phi\in\wt{\CC Q}_{\cy}$ be a potential of order $\geq 2$. Suppose $\Phi$ is $\wt{J}$-finite. Then
\begin{enumerate}
\item[$(1)$] $\Phi \in \wt{\pi}(\wt{J}(Q,\Phi))$ (i.e. $\Phi$ is quasi-homogeneous) if and only if $\Phi\in \pi(\wt{\fm}\cdot \wt{J}(Q,\Phi)+\wt{J}(Q,\Phi)\cdot \wt{\fm})$.
\item[$(2)$] For any potential $\Psi\in \wt{\CC Q}_{\cy}$ of order $\geq 2$ with $\wt{J}(Q,\Psi) = \wt{J}(Q,\Phi)$, it follows that $\Phi-\Psi\in \wt{\pi}(\wt{J}(Q,\Phi))$ if and only if $\Phi-\Psi\in \wt{\pi}(\wt{\fm}\cdot\wt{J}(Q,\Phi)+\wt{J}(Q,\Phi)\cdot \wt{\fm})$.
\end{enumerate}
\end{prop}

The following theorem is the analytic version of Theorem 3.16 of \cite{HuaZhou}. It provides a positive answer to the question proposed in Remark 3.10 of \cite{Dav19}.
\begin{theorem}\label{finite-determinacy}
Let $Q$ be a finite quiver and  $\Phi\in \wt{\CC Q}_{\cy}$ a potential. If $\Phi$ is $\wt{J}$-finite then $\Phi$ is finitely determined. More precisely, if
$\wt{J}(Q,\Phi) \supseteq \wt{\fm}^r$ for some integer $r\geq 0$ then $\Phi$ is $(r+1)$-determined.
\end{theorem}
\begin{proof}
Since $\Phi$ is $\wt{J}$-finite, there exists $r>0$ such that  $\wt{J}(Q,\Phi) \supseteq \wt{\fm}^r$. We proceed to show $\Phi$ is $(r+1)$-determined. 
It will follow that $\Phi$ is analytically right equivalent with $\Phi^{(r+1)}$.  
Suppose $\Psi\in \wt{\CC Q}_{\cy}$ such that $\Psi^{(r+1)}=\Phi^{(r+1)}$.  Let
\[
\Theta_t:= \Phi+ t(\Psi-\Phi) \in \wt{KQ}_{\cy},
\] where $K=K_U$ for some proper open subset set containing the line segment $[0,1]$.
Clearly, we have
\[
\wt{\Phi}_* \big (\wt{\cder}_l^+ (\wt{\CC Q}) \big ) \supseteq \wt{\fm}^{r+1}.
\]
Then Lemma \ref{tangent-compare-1} tells us
\[
\wt{\Phi}_* \big (\wt{\cder}_L^+ (\wt{KQ}) \big )   \supseteq \wt{\fn}^{r+1}.
\]
Since $\Theta_t$ and $\Phi$ have the same $(r+1)$-jet in $\wt{KQ}_{\cy}$, it follows readily that
\[
\wt{\Theta_t}_*\big (\wt{\cder}_{L} ^+(\wt{KQ}) \big) +\wt{\fn}^{r+2} = \wt{\Phi}_* \big (\wt{\cder}_{L} ^+ (\wt{KQ}) \big) +\wt{\fn}^{r+2}\supseteq \wt{\fn}^{r+1}.
\]
Then  the Nakayama lemma  tells us
\[
\wt{\Theta_t}_*\big (\wt{\cder}_{L} ^+(\wt{KQ}) \big) \supseteq \wt{\fn}^{r+1}.
\]
Consequently,
\[
\wt{\Theta_t}_{\#} \big (\der_{L} ^+ (\wt{KQ})  \big ) = \wt{\pi} \big(\wt{\Theta_t}_* \big (\wt{\cder}_{L} ^+(\wt{KQ})\big )\big )\supseteq \wt{\pi}(\wt{\fn}^{r+1}) \ni \Psi-\Phi = \frac{~d\, \Theta_t}{d~ t}.
\]
Suppose $\frac{~d\, \Theta_t}{d~ t}=\wt{\Theta_t}_\#(\xi)=\wt{\Theta_t}_\#(\wt{\mu}\circ\delta_t)$ for 
a double $L$-derivation
\[
\delta_t=\sum_{a\in Q_1}^n \sum_{u,v} A_{u,v}^{(a)}(t)~ u* \frac{\partial~}{\partial a} * v \in \wt{\dder}_L^+(\wt{KQ}),
\]
where $u$ and $v$ runs over paths with $t(u) = t(a)$ and $s(v) =s(a)$ respectively.  In $\wt{KQ}_{\cy}$, we have
\begin{eqnarray*}
\wt{\Theta_t}_\#(\xi)
&=& \pi\bigg ( \wt{\Theta_t}_*\big (\wt{\mu}\circ \wt{\tau} \circ \delta_t \big ) \bigg )\\
&=& \pi \big ( \sum_{a\in Q_1} \sum_{u,v} A_{u,v}^{(a)}(t)~ u\cdot  \wt{\Theta_t}_*(D_a) \cdot v \big )\\
&=& \pi \big ( \sum_{a\in Q_1} \sum_{u,v} A_{u,v}^{(a)}(t)~ vu\cdot  \wt{\Theta_t}_*(D_a) \big )\\
&=& \pi \big ( \sum_{a\in Q_1} \xi(a)\cdot  \wt{\Theta_t}_*(D_a) \big )
\end{eqnarray*}

Set $\bF=\xi$. For any $t_0\in [0,1]$,
the initial value problem
\begin{align*}\begin{cases}
\partial_t\bu=\bF(\bu,t) \\
\bu|_{t=t_0}=\id
\end{cases}
\end{align*}
admits a unique $V$-family of analytic solution $\bu=H_t(a)$ in a 
neighborhood $V$ containing $t_0$ by Proposition \ref{Moser}. Then $\{H_t|t\in V\}$ is an analytic family of automorphisms of $\wt{\CC Q}$. By the noncommutative chain rule 
(\cite[Lemma 2.11]{HuaZhou}) in $\wt{KQ}_{\cy}$,  we have
\begin{eqnarray*}
\frac{~~d\, H_t(\Theta_t)}{d~ t} &=& H_t(\frac{~d\, \Theta_t}{d~ t}) + \pi\bigg(\sum_{a\in Q_1}\frac{~d\, H_t(a)}{d~ t}\cdot H_t \big (\wt{\Theta_t}_*(D_a)\big) \bigg)  \\
&=&H_t \bigg ( \wt{\Theta_t}_\#(\xi)- \pi\big(\sum_{a\in Q_1} \xi(a) \cdot \wt{\Theta_t}_*(D_a) \big )\bigg)\\
&=&0.
\end{eqnarray*}
For $t_1\in V$, 
$H_{t_1}(\Theta_{t_1}) = H_{t_0}(\Theta_{t_0})$, i.e. $\Theta_{t_1}\sim_a \Theta_{t_0}$ in $\wt{\CC Q}$.
By compactness, we can cover $[0,1]\subset U$ by finitely many open subsets where in each open subset the evaluation of $\Theta_t$ are analytically right equivalent. Therefore, $\Phi=\Theta_0$ is analytically right equivalent to $\Psi=\Theta_1$.
\end{proof}

The following theorem is the analytic version of the noncommutative Mather-Yau theorem for formal potentials proved in \cite{HuaZhou}.

\begin{theorem}\label{ncMY}
Let $Q$ be a finite quiver and $\Phi, \Psi$ be two $\wt{J}$-finite analytic potentials of order $\geq 3$. Then the following are equivalent 
\begin{enumerate}
\item[$(1)$] There exists an $l$-algebra isomorphism $\gamma: \wt{\Lm}(Q,\Phi)\cong \wt{\Lm}(Q,\Psi)$ so that $\gamma_*([\Phi]_\Phi)=[\Psi]_\Psi$.
\item[$(2)$] $\Phi$ and $\Psi$ are analytically right equivalent.
\end{enumerate}
\end{theorem}

\begin{proof}
By Proposition \ref{Jacobi-transform}, it is easy to see that (2) implies (1). Next we  show the converse.

First we claim that $\gamma$ can be lifted to an $l$-algebra automorphism of $\wt{\CC Q}$. By abuse of notation, denote the images of the arrows $a\in Q_1$ in $\wt{\Lm}(Q,\Phi)$ and $\wt{\Lm}(Q,\Psi)$ both by $\ol{a}$. Fix a lifting $h_a\in e_{s(a)} \cdot \wt{\CC Q} \cdot e_{t(a)}$ of $\gamma(\ol{a})$ for every $a\in Q_1$. Then we have an $l$-algebra endomorphism $H: a\mapsto h_a$ of $\wt{\CC Q}$ which lifts  $\gamma$. 
In other words,  we have a commutative diagram of $l$-algebra homomorphisms:
\[
\xymatrix{
\wt{\CC Q}\ar[d]\ar[r]^H & \wt{\CC Q}\ar[d]\\
 \wt{\Lm}(Q,\Phi)\ar[r]^\gamma &  \wt{\Lm}(Q,\Psi).
}
\]
Define $\wt{\fm}_\Phi \subset \wt{\Lm}(Q,\Phi)$  to be $\wt{\fm}/ \wt{J}(Q, \Phi)$ and similarly for $\wt{\fm}_\Psi\subset \wt{\Lm}(Q,\Psi)$. Because $\Phi$ and $\Psi$ are of order  $\geq3$, there is a canonical isomorphism of $l$-bimodules  $\wt{\fm}/\wt{\fm}^2\cong \wt{\fm}_\Phi/\wt{\fm}_\Phi^2\cong \wt{\fm}_\Psi/\wt{\fm}_\Psi^2$.  Because $\gamma$ induces an isomorphism on $\wt{\fm}_\Phi/\wt{\fm}_\Phi^2\cong \wt{\fm}_\Psi/\wt{\fm}_\Psi^2$, $H$ induces an isomorphism on $\wt{\fm}/\wt{\fm}^2$. Thus $H$ is invertible by the inverse function theorem \ref{anal-inverse}. 

By the assumption $\gamma_*([\Phi]_\Phi) = [\Psi]_\Psi$ we have $[H(\Phi)]_\Psi=[\Psi]_\Psi$, and by Proposition \ref{Jacobi-transform} we have
\[
\wt{J}(Q,\Psi) = H(\wt{J}(Q,\Phi)) = \wt{J}(Q,H(\Phi)).
\]
Thus, without loss of generality, we may replace $\Phi$ by $H(\Phi)$ and assume 
a priori that
\[
\wt{J}(Q,\Phi) = \wt{J}(Q,\Psi) \quad \text{ and }\quad [\Phi]_\Phi=[\Psi]_\Psi.
\]

Let $r$ be the minimal integer so that $\wt{J}(Q,\Phi) \supseteq \wt{\fm}^r$.
By finite determinacy (Theorem \ref{finite-determinacy}), it suffices 
to show that $\Phi^{(s)}$ and $\Psi^{(s)}$ lie in the same orbit of $\cG^s=\Aut_l(\cJ^s)$ for $s=r+1$. If $\Phi^{(s)}=\Psi^{(s)}$, then there is nothing to prove. So we may assume further that $\Phi^{(s)}\neq \Psi^{(s)}$.

Since $\cJ^s_{\cy}:=\cJ^s/[\cJ^s,\cJ^s]$ is a finite dimensional vector space, it has a natural complex manifold structure. Also,  it is not hard to check that $\mathcal{G}^s$ is a complex Lie group which acts analytically on $\cJ^s_\cy$. So the orbit $\mathcal{G}^s\cdot \Xi^{(s)}$ is an immersed submanifold of $\cJ^s_\cy$ for any potential $\Xi\in \wt{\CC Q}_{\cy}$. We proceed to calculate $T_{\Xi^{(s)}}(\mathcal{G}^s\cdot \Xi^{(s)})$, the tangent space  of $\mathcal{G}^s\cdot \Xi^{(s)}$  at $\Xi^{(s)}$. Let $\der_l^+(\cJ^s)$ be the space of $l$-derivations of $\cJ^s$ satisfying that  $\delta(\wt{\fm}/ \wt{\fm}^{s+1}) \subseteq \wt{\fm}/ \wt{\fm}^{s+1}$. Clearly, the canonical map $\rho_s: \der_l^+ (\wt{\CC Q}) \to \der_l^+(\cJ^s)$ is surjective. We have a commutative diagram of vector spaces over $\CC$ as follows:
\[
\xymatrix{
\wt{\dder}_l^+ (\wt{\CC Q}) \ar@{->>}[r]^-{\wt{\mu}\circ \wt{\tau}\circ-}  \ar@{->>}[d]^-{\wt{\mu}\circ-} &   \wt{\cder}_l^+ (\wt{\CC Q}) \ar[r]^-{\wt{\Xi}_*} & \wt{\CC Q}\ar[d]^{\wt{\pi}} \\
\der_l^+ (\wt{\CC Q}) \ar@{->>}[d]^{\rho_s} \ar[rr]^-{\wt{\Xi}_{\#}} & & \wt{\CC Q}_{\cy}  \ar[d]^{q_s} \\ 
\der_l^+(\cJ^s) \ar[rr]^{(\Xi^{(s)})_{\#}} && \cJ^s_\cy,
}  
\]
where  $(\Xi^{(s)})_{\#}$ is constructed in Lemma \ref{cycprop}.  Recall that $\der_l^+(\cJ^s)$ is the tangent space of $\mathcal{G}^s$ at the identity map, we have
\[
T_{\Xi^{(s)}}(\mathcal{G}^s\cdot \Xi^{(s)}) = \im ((\Xi^{(s)})_{\#}) =q_s \bigg(\wt{\pi} \big(\wt{\fm}\cdot \wt{J}(Q,\Xi) + \wt{J}(Q,\Xi)\cdot \wt{\fm} \big)\bigg).
\]

Now consider the complex line $\cL:=\{~ \Theta_t^{(s)}=t\Psi^{(s)}+(1-t)\Phi^{(s)}~|~t\in\CC ~ \}$ contained in $\cJ^s_\cy$. By the assumption that $\wt{J}(Q,\Phi) = \wt{J}(Q,\Psi)$, we have
\[T_{\Psi^{(s)}} (\mathcal{G}^s \cdot \Psi^{(s)}) = T_{\Phi^{(l)}} (\mathcal{G}^s\cdot \Phi^{(s)})=q_s \bigg(\pi \big(\wt{\fm}\cdot \wt{J}(Q,\Phi) + \wt{J}(Q,\Phi) \cdot \wt{\fm} \big)\bigg),\]
as subspaces of $\cJ^s_\cy$.  It follows that for any $t$, the tangent space $T_{\Theta_t^{(s)}} (\mathcal{G}^s\cdot \Theta_t^{(s)})$ is a subspace of  $q_s \bigg(\pi \big(\wt{\fm}\cdot\wt{J}(Q,\Phi) + \wt{J}(Q,\Phi) \cdot \wt{\fm} \big)\bigg)$. Let $\cL_0$ be the subset of $\cL$ consisting of those $\Theta_t^{(s)}$ such that
\[
T_{\Theta_t^{(s)}} (\mathcal{G}\cdot \Theta_t^{(s)}) = q_s \bigg(\pi \big(\wt{\fm}\cdot \wt{J}(Q,\Phi) + \wt{J}(Q,\Phi)\cdot \wt{\fm} \big)\bigg).
\]
Then $\Phi$ and $\Psi$ are  both in $\cL_0$. It remains to show that  $\cL_0$ lies in the orbit  $\cG^s \cdot\Phi^{(s)}$. By a standard lemma in the theory of Lie groups 
(cf. Lemma 1.1 \cite{Wall}), it suffices to check that
\begin{enumerate}
\item[(1)] The complement $\cL\backslash \cL_0$ is a finite set (so $\cL_0$ is a connected smooth submanifold of $\cJ^s_\cy$).
\item[(2)] The dimension of  $T_{\Theta_t^{(s)}}(\cG^s \cdot \Theta_t^{(s)})$ does
not depend on the choice of $\Theta_t^{(s)}\in \cL_0$.
\item[(3)] For all $\Theta_t^{(s)}\in \cL_0$, the tangent space  $T_{\Theta_t^{(s)}}(\cL_0)$  is contained in $T_{\Theta_t^{(s)}}(\cG^s \cdot \Theta_t^{(s)})$.
\end{enumerate}
Condition (1) holds  because $\cL\backslash \cL_0$ corresponds to the locus of the
parameters $t\in\CC$ where the continuous family
\[\{(\Theta_{t}^{(s)})_{\#}: \der_l^+(\cJ^s) \to \cJ^s_\cy \}_{t\in \CC}\]
 of linear maps between two finite dimensional spaces has non maximal rank. 
Condition $(2)$ follows from the construction of $\cL_0$. Note that the tangent space of $\cL_0$ at each of its point  is spanned by $\Phi^{(s)}-\Psi^{(s)} = q_s(\Phi-\Psi)$ in $\cJ^s_\cy$. By Proposition \ref{Bootstrapping} (2), condition (3) holds if $\Phi-\Psi \in \wt{\pi}(\wt{J}(Q,\Phi))$,  which is equivalent to the assumption that $[\Phi]_\Phi=[\Psi]_\Psi$.

\end{proof}

\section{Donaldson-Thomas invariant of quiver with potential} \label{sec:DT}
\subsection{Moduli of finite dimensional modules over the Jacobi algebra}
Let $Q$ be a finite quiver and $\Phi\in \wh{\CC Q}$ be a potential. A finite dimensional representation $V$ of $Q$ is a finite dimensional (left) module over $\CC Q$.  In particular, to each $i\in Q_0$  we associate a finite dimensional vector space $V_i$ and to each $a\in Q_1$ we associate
a linear operator in $\Hom_\CC(V_{s(a)},V_{t(a)})$. Given a vector $\bv\in \NN^{|Q_0|}$, we denote by $\Rep_\bv(Q)$ the space of representations of $Q$ with dimension vector $\bv$. Clearly,
\[
\Rep_\bv(Q)\cong \prod_{a\in Q_1} \Hom(\CC^{v_{s(a)}},\CC^{v_{t(a)}}).
\]
Denote by $G_\bv$ the algebraic group $\Pi_{i\in Q_0} GL_{v_i}$. It acts on $\Rep_\bv(Q)$ by conjugation. The moduli stack of representations of $Q$ with dimension vector $\bv$ is defined to be
\[
\cM_\bv(Q):=[\Rep_\bv(Q)/G_\bv].
\]
It has a coarse moduli space, defined by the semi-simplification map
\[
p_\bv: \cM_\bv(Q)\to M_\bv(Q):=\Rep_\bv(Q)\sslash G_\bv.
\]

A (finite dimensional) representation $V$ is called \emph{nilpotent} if all $a\in Q_1$ act on $V$ by nilpotent matrices. More generally, for an algebra $A$ an  $A$-module $M$ is called nilpotent if there exists $N>0$ such that $\fm^N M=0$ where $\fm$ is the Jacobson radical of $A$. Nilpotent representations of $Q$ are precisely $\wh{\CC Q}$ modules. Denote by $\Rep_\bv^{np}(Q)\subset\Rep_\bv(Q)$ the subset of nilpotent representations. Clearly, it is stable under the $G_\bv$-action. Let $[0]\in M_\bv(Q)$ be the point represented by the isomorphism class of the semisimple module 
$0:=\bigoplus_{i\in Q_0} S_i\ot V_i$, where $S_i$ is the simple module
corresponding to the node $i$. Then 
\[
\Rep^{np}_\bv(Q)=p_\bv^{-1}([0]).
\] In other words, a representation is nilpotent if it is a finite iterated extension of 
$S_i$.

Fix a lift $\phi\in\wh{\CC Q}$ of $\Phi$. For any $\bv\in \NN^{|Q_0|}$, let 
\[
\phi_\bv: \prod_{a\in Q_1}\Hom_k(V_{s(a)},V_{t(a)})\to \End_k(\bigoplus_{i\in Q_0} V_i)
\]
be the formal series of matrix variables for $\phi$. It can be viewed as a matrix valued formal function at $0\in\Rep_\bv(Q)$. Then $\tr(\phi_\bv)$ is a scalar valued formal function at $0\in\Rep_\bv(Q)$, where $\tr$ is defined to be $\sum_{i\in Q_0} \tr_{\gl_{v_i}}$. Since this is independent of the choice of a lift, we will write $\Phi_\bv:=\tr(\phi_\bv)$. We call $\Phi_\bv$ the \emph{formal Chern-Simons function}. Note that $\Phi_\bv$ is 
$G_\bv$-invariant. 

\begin{prop}\label{Toda}
Given $\Phi\in \wt{\CC Q}_{\cy, C}$ and a dimension vector $\bv$, the Chern-Simons functional $\Phi_\bv$ is absolutely convergent in a neighborhood of $\Rep^{np}_\bv(Q)\subset \Rep_\bv(Q)$. For $C^\p>C$ and $\Phi^\p$ the image of $\Phi$ under the natural monomorphism $\wt{\CC Q}_{\cy,C}\to \wt{\CC Q}_{\cy,C^\p}$, 
we have $\Phi^\p_\bv=\Phi_\bv$ in some neighborhood of $\Rep^{np}_\bv(Q)$. Moreover, a point in $\Rep_\bv^{np}(Q)$ is a critical point of $\Phi_\bv$ if and only if it is a module over $\wt{\Lm}(Q,\Phi)$ with dimension vector $\bv$.
\end{prop}
\begin{proof}
The first half of the proposition has already been proved by Toda (see Lemma 2.15 \cite{Tod17}). We include a proof just for the convenience of the reader. 
Set $m=|Q_0|$, $k=|Q_1|$ and $\bv=(v_1,\ldots,v_m)$. A representation $V$ of $Q$ with dimension vector $\bv$ corresponds to a family  $(A_a)_{a\in Q_1} \in \prod_{a\in Q_1}\Hom_k(V_{s(a)},V_{t(a)})$, where  $V_i\cong \CC^{v_i}$. Given a noncommutative series $f\in\wh{\CC Q}$, denote by $(A_a) \mapsto f_\bv(A_a)$ 
the matrix valued formal map from $\prod_{a\in Q_1}\Hom_k(V_{s(a)},V_{t(a)})$ to 
$\End_k(\bigoplus_{i=1}^m V_i)$. We fix a  lift $\phi\in \wt{\CC Q}_C$ for $\Phi$. Then $\Phi_\bv=\tr(\phi_\bv(A_a))$.  Write
\[
||A_a||=\left(\sum_{i,j} |A^{ij}_a|^2\right)^{1/2}
\] 
for the Frobenius norm. For $\ep>0$, let $\Delta_\ep:=\{ (A_a)_{a\in Q_1}| ||A_a||<\ep\}$. Since any nilpotent orbit contains $0$ in its closure, the
set $U_{\bv,\ep}:=G_\bv\cdot \Delta_\ep$ contains $\Rep^{np}_\bv(Q)$. 
If $\ep<\frac{1}{kC}$ then 
\[
\left| a_w \cdot \tr~ \prod_{w=a_1\ldots a_d} A_{a_1}\ldots A_{a_d}\right|\leq C^{|w|} \prod_{i=1}^{|w|} ||A_{a_i}||=(\ep C)^{|w|},
\] and 
\[
||\Phi_\bv||< \sum_{|w|} (k\ep C)^{|w|}<+\infty,
\]
i.e. $\Phi_\bv$ is absolutely convergent in $U_{\bv,\ep}$. For $C^\p>C$, $\Phi^\p_\bv$ is clearly the restriction of $\Phi_\bv$ in the neighborhood with $\ep<\frac{1}{k C^\p}$.

By definition, $V=(V_i, A_a| i\in Q_0, a\in Q_1)$ is a representation of $\wt{\CC Q}$ if and only if for any $f\in \wt{\CC Q}$ the formal map $f_\bv(A_a)$ is convergent. We will see that $V$ must be nilpotent (Lemma \ref{lem-nilp}).  Let $E_a^{ij}\in \Hom_k(V_{s(a)},V_{t(a)})$ be the elementary matrix for a choice of basis. Then 
\begin{align*}
\nabla \Phi_\bv(A_a|a\in Q_1)=0 ~~&\Leftrightarrow ~~\lim_{\delta\to 0} \frac{\tr~ \phi_\bv(A_b+\delta E^{ij}_{b})-\tr~ \phi_\bv(A_a)}{\delta}=0  ~~~\text{for any $b\in Q_1$ and $i,j$}\\
 & \Leftrightarrow ~~\tr~\left( (D_b\Phi)_\bv(A_a)\cdot E^{ij}_b \right)=0 ~~~\text{for any $b\in Q_1$ and $i,j$}\\
 & \Leftrightarrow ~~(D_b\Phi)_\bv(A_a)=0.
\end{align*}
Equivalently,  the image of 
\[ 
\wt{\Phi}_*: \wt{\cder}_l(\wt{\CC Q})\to \wt{\CC Q}
\] 
vanish at the point $(A_a)_{a\in Q_1}$, i.e. when $(A_a)_{a\in Q_1}$ defines a $\wt{\Lm}(Q,\Phi)$-module.
\end{proof}

\begin{lemma}\label{lem-nilp}
Any finite dimensional module over $\wt{\CC Q}$ is nilpotent. 
\end{lemma}
\begin{proof}
 Let $\wt{\fm}$ be the two-sided ideal of the analytic Jacobi algebra 
 $\wt{\CC Q}$ generated by the arrows. Then every element of $1+\wt{\fm}$ is invertible
 by Lemma~\ref{lemma:exponential-inverse}.
 Hence $\wt{\fm}$ is contained in the Jacobson radical, i.e. the intersection of the 
 annihilators of all the simple modules. Since the quotient of $\wt{\CC Q}$ by
 $\wt{\fm}$ is semi-simple, the ideal $\wt{\fm}$ in fact equals the Jacobson ideal.
 In particular, the ideal $\wt{\fm}$ acts nilpotently on each finite-dimensional module.
\end{proof}

Although the natural map $\wt{\Lm}(Q,\Phi)\to \wh{\Lm}(Q,\Phi)$ for a general analytic potential $\Phi$ might be neither injective nor surjective, they have isomorphic categories of finite dimensional modules.
\begin{prop}\label{anal-formal}
Let $\Phi$ be an analytic potential. Then the natural morphism $\wt{\Lm}(Q,\Phi)\to \wh{\Lm}(Q,\Phi)$ induces an isomorphism between the categories of
finite dimensional modules.
\end{prop}
\begin{proof}
For simplicity, we write $\wt{\Lm}=\wt{\Lm}(Q,\Phi)$, $\wh{\Lm}=\wh{\Lm}(Q,\Phi)$, 
$\wt{J}=\wt{J}(Q,\Phi)$ and $\wh{J}=\wh{J}(Q,\Phi)$. Let $M$ be a finite dimensional module over $\wt{\Lm}$. By Lemma \ref{lem-nilp}, $M$ is nilpotent over $\wt{\CC Q}$. Since $\wh{J}\subset \wt{J}+\wh{\fm}^N$ for any $N>0$,  we obtain that $M$ is a 
$\wh{\Lm}$ module. As a consequence, the restriction along $\wt{\Lm} \to \wh{\Lm}$
induces a bijection on the class of finite-dimensional modules.
The restriction along $\wt{\Lm}\to \wh{\Lm}$ also induces a 
bijection in the morphism spaces, since a $\CC$-linear map between
finite-dimensional modules commutes with the action of the arrows of $Q$ iff
it is $\wh{\CC Q}$-linear iff it is $\wt{\CC Q}$-linear. Therefore, the natural functor 
is an isomorphism between the categories of finite dimensional modules. 
\end{proof}

\begin{theorem}\label{pervsheaf-QP}
Let $Q$ be a quiver and $\Phi$ be an analytic potential on $Q$. The moduli stack of finite dimensional $\wh{\Lm}(Q,\Phi)$-modules is equipped with a canonical perverse sheaf of vanishing cycles.
\end{theorem}
\begin{proof}
Modules of $\wh{\Lm}(Q,\Phi)$ with fixed dimension form an algebraic stack. We refer to Section 6 of \cite{Nagao} for the construction. 
By Lemma \ref{lem-nilp} and Proposition \ref{anal-formal}, the moduli stack of finite dimensional $\wh{\Lm}(Q,\Phi)$-modules coincides with the moduli stack of  finite dimensional $\wt{\Lm}(Q,\Phi)$-modules, which embeds into the moduli stack of finite dimensional nilpotent representations of $Q$ by Lemma \ref{lem-nilp}. Moreover, it is the intersection of the critical stack $\left[\{\nabla \Phi_\bv=0\}/G_\bv\right]$ and $[\Rep^{np}_\bv(Q)/G_\bv]$ when the dimension vector is $\bv$. The restriction of the $G_\bv$-equivariant perverse sheaf of vanishing cycles defines the desired perverse sheaf.
\end{proof}

The following corollary is an immediate consequence of Theorem \ref{pervsheaf-QP} and Theorem \ref{ncMY}.
\begin{corollary}\label{cor:perv-J-finite}
Let $Q$ be a finite quiver and $\Phi, \Psi$ two $\wt{J}$-finite analytic potentials
 of order $\geq 3$. Suppose 
\[
\gamma: \wt{\Lm}(Q,\Phi)\cong \wt{\Lm}(Q,\Psi)
\] is a $\CC Q_0$-algebra isomorphism such that $\gamma_*([\Phi]_\Phi)=[\Psi]_\Psi$. Then $\gamma$ induces an isomorphism between the canonically defined perverse sheaves of vanishing cycles. If $\Phi$ is further assumed to be quasi-homogeneous, then the isomorphism class of the perverse sheaf is determined by the 
isomorphism class of the Jacobi algebra $\wt{\Lm}(Q,\Phi)$.
\end{corollary}

\subsection{Motivic Hall algebra and integration map}\label{sec:integration}
The main purpose of this subsection is to define Donaldson-Thomas invariants for 
quivers with analytic potentials. Our setup is parallel to \cite{Nagao}, where the topological Euler characteristic version of DT invariants is studied. The novelty of this paper is that the Behrend function is taken into account. The case of algebraic potentials has already been considered in \cite{JS08} and \cite{KS08}. However, we should emphasize that even for algebraic potentials the DT theory of the algebraic Jacobi algebra and the DT theory of the formal Jacobi algebra are not the same. For the applications to cluster algebras, we need to consider the formal Jacobi algebra while only algebraic Jacobi algebras were considered in \cite{JS08} and \cite{KS08}. It turns out that the analytic potentials provide a more flexible framework so that the invariants are still well defined after mutations (see next section). 

The foundations for the study of the 
motivic Hall algebra were established in \cite{JS08} and \cite{KS08}. We refer  to 
Section 7 of \cite{Nagao} for the precise definitions. Fix a finite quiver $Q$ with $n$ nodes and an analytic potential $\Phi$. Denote by $\wh{\Lm}:=\wh{\Lm}(Q,\Phi)$ the 
formal Jacobi algebra. The motivic Hall algebra of the category of finite dimensional $\wh{\Lm}$-modules is denoted by $\MH_{\wh{\Lm}}$. It is equipped with an associative product $\star$. The Hall algebra is graded by $\NN^n$:
\[
\MH_{\wh{\Lm}}=\bigoplus_{\bv\in \NN^n} \MH_{\wh{\Lm}}(\bv).
\]
The elements of $\MH_{\wh{\Lm}}(\bv)$ are 
equivalence classes of stack morphisms $f: X\to \mod_{\bv}-\wh{\Lm}$, 
where the latter is the moduli stack of $\wh{\Lm}$-modules of dimension vector $\bv$. By an abuse of notation, we use the same symbol to denote the category as well as the classifying stack of objects in the category. Since $\MH_{\wh{\Lm}}$ is an algebra over the Grothendieck ring of $\CC$-stacks $\cK(St/\CC)$, we may consider the $\cK(Var/\CC)[\LL^{-1}]$ submodule generated by those morphisms where $X$ is a variety, which will be denoted by $\MH_{\wh{\Lm},0}$. Indeed it is a subring of $\MH_{\wh{\Lm}}$. We call an element in $\MH_{\wh{\Lm}}$ that lies in $\MH_{\wh{\Lm},0}$ \emph{regular}. Define
\[
\MH_{\wh{\Lm},sc}:= \MH_{\wh{\Lm},0}/(\LL-1)\MH_{\wh{\Lm},0}
\] and call it the \emph{semiclassical limit of $\MH_{\wh{\Lm}}$}.
One can show that $\MH_{\wh{\Lm},sc}$ is a commutative Poisson algebra (\cite[Theorem 7.2]{Nagao}) with Poisson bracket:
\[
\{f,g\}=\frac{f\star g- g\star f}{\LL-1}.
\]

We put 
\[
T_Q:=\CC[y_1^{\pm 1},\ldots,y_n^{\pm 1}], ~~~T_Q^\vee:=\CC[x_1^{\pm 1},\ldots,x_n^{\pm 1}], ~~~\TT_Q:=T_Q^\vee\ot_\CC T_Q.
\]
Set $Q_0=\{1,\ldots,n\}$ and $\Gamma_{Q,\Phi}$ to be the Ginzburg algebra for $(Q,\Phi)$ (see Section 1.1 \cite{Nagao}). Let $s_i$ be the simple module for $i$ and $\Gamma_i$ be $e_i\Gamma_{Q,\Phi}$. We will fix an isomorphism between $T_Q$ and the group algebra of the numerical Grothendieck group of $\D_{fd}(\Gamma_{Q,\Phi})$ such that $y_i=\by^{[s_i]}$. And we set $x_j:=\bx^{[\Gamma_j]}$.
The quantum torus $\QT$ is defined to be the group algebra of the numerical Grothendieck group of $\D_{fd}(\Gamma_{Q,\Phi})$ over the coefficient field $\CC(t)$. To be more specific, 
\[
\QT=\sum_{\bv\in \ZZ^n} \CC(t)\cdot \by^{\bv}
\] where $\by^{\bv}=\prod_{i=1}^n y_i^{v_i}$. Denote by $\QT_0$ the $\CC[t^{\pm 1}]$-subalgebra of $\QT$ generated by the $\by^{\bv}$. We put
\[
\QT_{sc}:=\QT_0/(t-1)\QT_0.
\]
The \emph{topological integration map} is the $\ZZ^n$-graded linear map 
\[
I: \MH_{\wh{\Lm},sc}\to  \QT_{sc}
\] defined by 
\[
I\left([f: X\to \mod_\bv-\wh{\Lm}]\right):= \chi(X)\cdot \by^\bv,
\] where $\chi(X)$ is the topological Euler characteristic of $X$.
Joyce and Song (in  \cite{JS08}) have proved that $I$ is a Poisson morphism with respect to the Poisson structures on $\MH_{\wh{\Lm},sc}$ (defined in Section 7.1.3 of \cite{Nagao}) and the Poisson structure on $\QT_{sc}$ defined by 
\[
\{\by^{\bv_1},\by^{\bv_2}\}:=\ol{\chi}(\bv_1,\bv_2)\cdot \by^{\bv_1+\bv_2}, 
\]
where $\ol{\chi}$ is the Euler pairing  on $\D_{fd}(\Gamma_{Q,\Phi})$.

Given a $\ZZ$-valued 
constructible function $\nu$ on $\mod_\bv-\wh{\Lm}$, the \emph{weighted integration map} 
\[
I^\nu: \MH_{\wh{\Lm},sc}\to  \QT_{sc}
\] is 
defined by 
\[
I^\nu\left([f: X\to \mod_\bv-\wh{\Lm}]\right):= \chi(X, f^*\nu)\cdot \by^\bv,
\] where $\chi(X, f^*\nu)=\sum_{n\in\ZZ} n\cdot\chi((f^*\nu)^{-1}(n))$.

By Proposition \ref{anal-formal}, $\mod_\bv-\wh{\Lm}$ is equivalent to $\mod_\bv-\wt{\Lm}$ as stacks. By Proposition \ref{Toda}, there is an embedding of analytic stacks 
\[\xymatrix{
\mod_\bv-\wh{\Lm}=\left[\left(\{\nabla \Phi_\bv=0\}\cap \Rep^{np}_\bv(Q)\right)/G_\bv\right]\ar[rrr]^-{j_C} &&& \left[\left(\{\nabla \Phi_\bv=0\}\right)/G_\bv\right].
}\]
 For any $\bv$, let $\nu$ be the $j_C$-pullback of the Behrend function (for the definition of the Behrend function on a stack, see Section 4 of \cite{JS08}). It is a constructible function on the moduli stack $\mod_{fd}-\wh{\Lm}$, independent of $C$ when $C\gg 0$. We will simply write $j$ for $j_C$.

The following theorem is essentially due to Joyce and Song. We include a proof because the setup in \cite{JS08} is slightly different. Our moduli stack $\mod_\bv-\wh{\Lm}$ is not locally a critical stack. Instead, it embeds into a critical stack as a closed substack.
\begin{theorem}(\cite[Theorem 5.11]{JS08})\label{integration}
 $I^\nu$ is a Poisson algebra morphism with the Poisson structure on $\QT_{sc}$ defined to be
\[
\{\by^{\bv_1},\by^{\bv_2}\}:=(-1)^{\ol{\chi}(\bv_1,\bv_2)}\ol{\chi}(\bv_1,\bv_2)\cdot \by^{\bv_1+\bv_2}.
\]
\end{theorem}
\begin{proof}
It is well known (cf. Theorem 5.2 \cite{Bri}) that
to prove that $I^\nu$ is a Poisson algebra morphism it suffice to verify that $\nu$ satisfies the identities in Theorem 5.11 \cite{JS08}, i.e. for all $E_1, E_2\in \mod_{fd}-\wh{\Lm}$
\begin{equation}\label{Bid1}
\nu(E_1\op E_2)=(-1)^{\ol{\chi}(E_1,E_2)} \nu(E_1)\nu(E_2),
\end{equation}
\begin{align}\label{Bid2}
\int_{[F]\in \PP\Ext^1(E_2,E_1)} \nu(F) d\chi -\int_{[\wt{F}]\in \PP\Ext^1(E_1,E_2)} \nu(\wt{F}) d\chi \\ \notag
=\left(\dim \Ext^1(E_2,E_1)-\dim \Ext^1(E_1,E_2) \right) \nu(E_1\op E_2).
\end{align}

Suppose the dimension vectors of $E_1$ and $E_2$ are $\bv_1=(v_{1,1},\ldots,v_{1,n})$ and $\bv_2=(v_{2,1},\ldots,v_{2,n})$. For simplicity, we denote by $\cM_{\wh{\Lm}}(\bv)$ the moduli stack of $\wh{\Lm}$-modules with dimension vector $\bv$. By Proposition \ref{Toda}, there exists a $G_\bv$-invariant analytic neighborhood $U_\bv\supset \Rep^{np}_\bv(Q)$ such that $\Phi_\bv$ is convergent on $U_\bv$, and there is an embedding of analytic stacks:
\[
j: \cM_{\wh{\Lm}}(\bv)\to \left[Z_{\Phi_\bv}/G_\bv\right]
\] where $Z_{\Phi_\bv}$ is the critical scheme of $\Phi_\bv$. The image of $j$ coincides with $\left[Z_{\Phi_\bv}\cap \Rep^{np}_\bv(Q)/G_\bv\right]$. Set $\bv=\bv_1+\bv_2$. Let $[E_1\op E_2]$ be the point in $\cM_{\wh{\Lm}}(\bv)$ represented by the $\wh{\Lm}$-module $E_1\op E_2$, and let $p=j([E_1\op E_2])$ be its image. By an abuse of notation, we also use $p$ to denote a critical point that is represented by $E_1\op E_2$. Note that this is uniquely defined up to $G_\bv$-action.
Recall that the Behrend function of the global quotient stack is defined to be
\[
\nu_{[Z_{\Phi_\bv}/G_\bv]}(p)=(-1)^{\dim G_\bv}\nu_{Z_{\Phi_\bv}}(p)=(-1)^{\dim G_\bv+\dim U_\bv}(1-\chi(MF_{\Phi_\bv}(p)))
\] where $MF_{\Phi_\bv}(p)$ is the Milnor fiber of the analytic function $\Phi_\bv$ at the critical point $p$. For simplicity, we write $\nu_{\Phi_\bv}=\nu_{Z_{\Phi_\bv}}$.

Note that $\Rep_\bv(Q)$ can be identified with 
\[
\fg:=\Ext^1(\bigoplus_{i=1}^n S_i^{\op v_i},\bigoplus_{i=1}^n S_i^{\op v_i}),
\]
which has a splitting
\[
\fg=\fg_{11}\op \fg_{22}\op \fg_{12}\op \fg_{21}
\] where $\fg_{ii}=\Rep_{\bv_i}(Q)$ for $i=1,2$ and 
\[
\fg_{12}=\Ext^1(\bigoplus_{i=1}^n S_i^{\op v_{1,i}},\bigoplus_{i=1}^n S_i^{\op v_{2,i}}), ~~~\fg_{21}=\Ext^1(\bigoplus_{i=1}^n S_i^{\op v_{2,i}},\bigoplus_{i=1}^n S_i^{\op v_{1,i}}).
\]
We remark that the trivial module $\bigoplus_{i=1}^n S_i^{\op v_i}$ is mapped to $0$ by $j$. Let $T$ be the one parameter subgroup of $G_\bv$ that fixes $\bigoplus_{i=1}^n S_i^{\op v_{1,i}}$ but scales $\bigoplus_{i=1}^n S_i^{\op v_{2,i}}$ by $\lambda\in \CC^\times$. Denote by $\fg^T$ the fixed locus of the $T$-action on $\fg$. Clearly,
\[
\fg^T=\fg_{11}\times \fg_{22}\times \{0\}\times \{0\}.
\]
Since $U_\bv$ is $G_\bv$-invariant, it is $T$-invariant in particular. Moreover, the fixed loci of the critical scheme of $\Phi_\bv$ can be identified with the critical scheme of the restriction of $\Phi_\bv$ on the fixed subspace, i.e. $Z_{\Phi_\bv}^T=Z_{\Phi_\bv|_{\fg^T}}$. Let $p_{ii}\in \fg_{ii}$ be a point representing the $\wh{\Lm}$-module $E_i$ for $i=1,2$.  We have $p=(p_{11},p_{22}, 0, 0)\in \fg^T$. The $G_{\bv_i}$-orbit of $p_{ii}$, or the point in the quotient stack is uniquely defined. By the Thom-Sebastiani theorem, we have
\[
\nu_{\Phi_\bv|_{\fg^T}}(p)=\nu_{\Phi_{\bv_1}}(p_{11})\cdot \nu_{\Phi_{\bv_2}}(p_{22}).
\] By localization at the $T$-fixed points, we have
\[
\nu_{\Phi_\bv}(p)=(-1)^{\dim \fg_{12}+\dim \fg_{21}}\nu_{\Phi_\bv|_{\fg^T}}(p).
\]
Since $\nu$ is defined to be the pull back of the Behrend function, we have
\begin{align*}
\nu(E_1\op E_2)&=(-1)^{\dim G_\bv}\nu_{\Phi_\bv}(p)=(-1)^{\dim G_\bv+\dim \fg_{12}+\dim \fg_{21}}\nu_{\Phi_{\bv_1}}(p_{11})\cdot \nu_{\Phi_{\bv_2}}(p_{22})\\
&=(-1)^{\dim G_\bv +\dim \fg_{12}+\dim \fg_{21}+\dim G_{\bv_1}+\dim G_{\bv_2}} \nu(E_1)\cdot \nu(E_2)
\end{align*}
Since $\mod_{fd}\wh{\Lm}$ is the heart of a t-structure on 
the finite derived category of the Ginzburg algebra of $\Gamma_{Q,\Phi}$, we have
\[
\Ext^d(E_i,E_j)=\Ext^d_{\Gamma_{Q,\Phi}}(E_i,E_j) ~~~~\text{for $d=0,1$.}
\] 
Therefore,  
\begin{align*}
&(-1)^{\dim G_\bv +\dim \fg_{12}+\dim \fg_{21}+\dim G_{\bv_1}+\dim G_{\bv_2}}\\
&=(-1)^{\dim\Hom(E,E)+\dim \Ext^1(E_1,E_2)+\dim \Ext^1(E_2,E_1)+\dim \Hom(E_1,E_1)+\dim \Hom(E_2,E_2)}\\
&=(-1)^{\dim \Hom(E_1, E_2)+\Ext^1(E_1,E_2)+\Ext^1(E_2,E_1)+\Hom(E_2,E_1)}\\
&=(-1)^{\ol{\chi}(E_1,E_2)}.
\end{align*}
The last equality follows from the Serre duality on the finite derived category of $\Gamma_{Q,\Phi}$. So we have proved identity \ref{Bid1}.

Let $F$ be the module that fits into the short exact sequence
\[\xymatrix{
0\ar[r] & E_1 \ar[r] & F\ar[r] & E_2\ar[r] & 0
}\] corresponding to the class $[F]\in \PP\Ext^1(E_2,E_1)$. Let $q:=(p_{11}, p_{22}, 0, p_{21})\in \fg$ be a point representing $F$. Similarly, let $\wt{q}:=(p_{11},p_{22},p_{12},0)$ be a point representing $\wt{F}$ with  
\[\xymatrix{
0\ar[r] & E_2 \ar[r] & \wt{F}\ar[r] & E_1\ar[r] & 0.
}\]
Then we have
\[
\nu(F)=\nu_{[Z/G_\bv]}(q)=(-1)^{\dim G_\bv} \nu_{\Phi_\bv}(q) 
\] and 
\[
\nu(\wt{F})=\nu_{[Z/G_\bv]}(\wt{q})=(-1)^{\dim G_\bv} \nu_{\Phi_\bv}(\wt{q}).
\]
By the identity \ref{Bid1} and the definition of Behrend function, 
to prove the identity \ref{Bid2}, it suffices to show that
\begin{align*}
&\int_{[p_{21}]\in \PP\Ext^1(E_2,E_1)} \left(1-\chi(MF_{\Phi_\bv}(q))\right) d\chi - \int_{[p_{12}]\in \PP\Ext^1(E_1,E_2)} \left(1-\chi(MF_{\Phi_\bv}(\wt{q}))\right) d\chi\\
&=\left(\dim \Ext^1(E_2,E_1)-\dim \Ext^1(E_1,E_2)\right)\cdot (1-\chi(MF_{\Phi_{\bv_1}+\Phi_{\bv_2}}(p_{11},p_{22}))).
\end{align*}
This is proved by Joyce and Song in Section 10.2 of \cite{JS08} using the blow-up formula for the Milnor number (Theorem 4.11 \cite{JS08}).
\end{proof}

For a fixed stability condition $\sigma$, we may define the \emph{Donaldson-Thomas invariant} of $\sigma$-semistable modules of $\wh{\Lm}(Q,\Phi)$ by taking the 
$I^\nu$-image of certain element in $\MH_{\wh{\Lm},sc}$ 
determined by the moduli stack of  $\sigma$-semistable modules. 
Since we will not use them in this paper, we refer the readers to Definition 7.15 of \cite{JS08} for details.

\section{Mutation and transformation of Donaldson-Thomas invariants}\label{sec:mutation} 

\subsection{Mutation of quivers with analytic potential}
We follow Derksen-Weyman-Zele\-vinsky's
fundamental article \cite{DWZ}. Let $Q$ be a finite
quiver. Denote by $l$ the subalgebra $\CC Q_0$. Let $\wh{\CC Q}$ be the completed path algebra. The {\em continuous zeroth Hochschild homology} of $\wh{\CC Q}$ is the vector
space $\HoH_0(\wh{\CC Q})$ obtained as the quotient of $\wh{\CC Q}$ by the closure of
the subspace generated by all commutators, i.e. $\HoH_0(\wh{\CC Q})=\wh{\CC Q}_\cy$. It admits a topological
basis formed by the {\em cycles} of $Q$.
In particular, the space $\wh{\CC Q}_\cy$ is a product of copies of $\CC$ indexed
by the vertices if $Q$ does not have oriented cycles.
A potential $\Phi\in\wh{\CC Q}_\cy$ is {\em reduced} if  
$\ord(\Phi)\geq 3$.  If the potential $\Phi$ is reduced and the Jacobian
algebra $\wh{\Lm}(Q,\Phi)$ is finite-dimensional, its associated quiver is isomorphic to $Q$.

As typical examples, we may consider the quiver $Q$
\begin{equation} \label{eq:cyclic-quiver}
\xymatrix{ & 2 \ar[dr]^a & \\
1 \ar[ur]^b & & 3 \ar[ll]^c }
\end{equation}
with the potential $\Phi=abc$ or with the potential $\Phi=(abc)^2$.

In order to define the mutation of a quiver with potential $(Q,\Phi)$
at a vertex $k$, we need to recall the construction of a reduced
quiver with potential from an arbitrary quiver with potential.
Two quivers with potential $(Q,\Phi)$ and $(Q',\Phi')$
are {\em right equivalent} if $Q_0=Q_0'$ and there exists a $l$-algebra
isomorphism $H: \wh{\CC Q} \to \wh{\CC Q'}$ such  the induced
map in topological Hochschild homology takes $\Phi$ to $\Phi'$.
A quiver with potential $(Q,\Phi)$ is {\em trivial} if $\Phi$ is
a linear combination of $2$-cycles
and $\wh{\Lm}(Q,\Phi)$ is isomorphic to $l$.
If $(Q,\Phi)$ and $(Q',\Phi')$ are two quivers with potentials
such that the sets of vertices of $Q$ and $Q'$ coincide,
their {\em direct sum} $(Q,\Phi)\oplus (Q',\Phi')$ is defined
as the pair consisting of the quiver with the same
vertex set, with set of arrows the disjoint union of
those of $Q$ and $Q'$, and with the potential equal
to the sum $\Phi + \Phi'$.

\begin{theorem}[\cite{DWZ}, Theorem~4.6 and Proposition~4.5]
Any quiver with potential $(Q,\Phi)$ is right equivalent to the direct
sum of a reduced one $(Q_{red},\Phi_{red})$ and a trivial one
$(Q_{triv},\Phi_{triv})$, both unique up to right equivalence. Moreover,
the inclusion induces an isomorphism from $\wh{\Lm}(Q_{red},\Phi_{red})$ onto 
$\wh{\Lm}(Q,\Phi)$.
\end{theorem}
The quiver with potential $(Q_{red}, \Phi_{red})$ is
the {\em reduced part} of $(Q,W)$. {\bf{Warning:}} $Q_{red}$ might contain 2-cycles.

We can now define the mutation of a quiver with potential.
Let $(Q,\Phi)$ be a quiver with potential such that $Q$ does not
have loops. Let $k$ be a vertex of $Q$
not lying on a $2$-cycle. The {\em mutation} $\mu_k(Q,\Phi)$ is
defined as the reduced part of the {\em pre-mutation}, i.e. of the quiver with potential
$\tilde{\mu}_k(Q,\Phi)=(Q',\Phi')$, which is defined as follows:
\begin{itemize}
\item[a)]
\begin{itemize}
\item[(i)] To obtain $Q'$ from $Q$, add a new arrow $[\alpha\beta]$
for each pair of arrows $\alpha:k\to j$ and $\beta: i \to k$ of $Q$ and
\item[(ii)] replace each arrow $\gamma$ with source or target $k$ by
a new arrow $\gamma^*$ with $s(\gamma^*)=t(\gamma)$ and $t(\gamma^*)=s(\gamma)$.
\end{itemize}
\item[b)] Put $\Phi'=\underline{\Phi}+\Delta$, where
\begin{itemize}
\item[(i)] $\underline{\Phi}$ is obtained from $\Phi$ by
replacing, in a representative of $\Phi$ given as an infinite linear combination
of paths non of which starts or ends at $k$, each occurrence of $\alpha\beta$ by 
$[\alpha\beta]$, for each pair
of arrows $\alpha:i\to k$ and $\beta: k \to j$ of $Q$;
\item[(ii)] $\Delta$ is the sum of the cycles $[\alpha\beta]\beta^* \alpha^*$
taken over all pairs of arrows $\alpha:k\to j$ and $\beta: i \to k$ of $Q$.
\end{itemize}
\end{itemize}
If $k$ is not contained in a $2$-cycle of $\mu_k(Q,\Phi)$ then
$\mu_k(\mu_k(Q,\Phi))$ is right equivalent to $(Q,\Phi)$ (Theorem~5.7 of \cite{DWZ}).
As examples, consider the mutation at $2$ of the cyclic quiver (\ref{eq:cyclic-quiver})
endowed with the potential $\Phi=abc$ and with $\Phi'=(abc)^2$. For $\Phi=abc$,
the mutated quiver with potential is the acyclic quiver
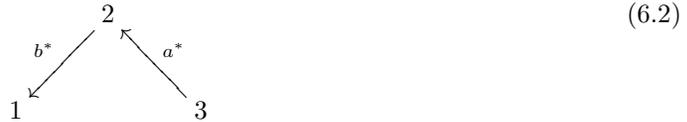
\begin{equation}
\xymatrix{ & 2 \ar[dl]_{b^*} & \\
1  & & 3 \ar[ul]_{a^*} }
\end{equation}
with the zero potential. But for $\Phi'=(abc)^2$, the mutated quiver with potential
is
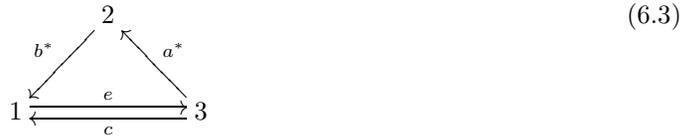
\begin{equation} \label{eq:2-cycle-after-mutation}
\xymatrix{ & 2 \ar[dl]_{b^*} & \\
1 \ar@<0.5ex>[rr]^e & & 3 \ar[ul]_{a^*} \ar@<0.5ex>[ll]^c }
\end{equation}
with the potential $ecec+eb^* a^*$.

The general construction implies that if neither $Q$ nor the quiver $Q'$ in
$(Q',\Phi')=\mu_k(Q,\Phi)$ have loops or $2$-cycles, then $Q$ and $Q'$ are linked
by the quiver mutation rule (cf. Prop.~7.1 of \cite{DWZ}) .
Thus, if we want to `extend' this rule to quivers with potentials, it is important
to ensure that no $2$-cycles appear during the mutation process.

Let $Q$ be a finite quiver. A {\em polynomial function} on $\wh{\CC Q}_\cy$
is the composition of a polynomial function on a finite-dimensional
vector space $V$ with a continuous linear map $\wh{\CC Q}_\cy \to V$. 
A {\em hypersurface} in
$\wh{\CC Q}_\cy$ is the set of zeroes of a non zero
polynomial function.

\begin{theorem}[\cite{DWZ}, Cor.~7.4] Let $Q$ be a
finite quiver without loops nor $2$-cycles. There is a countable
union of hypersurfaces $\Delta\subset \wh{\CC Q}_\cy$ such that
for each $\Phi$ not belonging to $\Delta$, no $2$-cycles appear
in any iterated mutation of $(Q,\Phi)$.
\end{theorem}

A potential $\Phi$ not belonging to $\Delta$ is called {\em nondegenerate}.
So if $Q$ is a quiver without loops nor $2$-cycles and
$\Phi$ a nondegenerate potential, we can indefinitely mutate the
quiver with potential $(Q,\Phi)$ and the mutation of the
underlying quivers is given by the quiver mutation rule.
Notice that the potential $\Phi=(abc)^2$ on the quiver~(\ref{eq:cyclic-quiver})
is degenerate, which is compatible with the appearance of
a $2$-cycle in (\ref{eq:2-cycle-after-mutation}).

The following proposition is a noncommutative version of the separation lemma for analytic functions. The formal analogue has been proved by Derksen, Weyman and Zelevinsky (see Lemma 4.7 \cite{DWZ}).

\begin{prop}\label{separation} 
Let $Q$ be a finite quiver with two distinct arrows $y,z$ such that $s(y)=t(z), s(z)=t(y)$. Let $Q_1=\{y,z,x_3,\ldots,x_k\}$. Let $\Phi$ be a potential of the form
\[
\Phi=yz-F(y,z,x),
\] where $F(y,z,x)$ is an analytic series in $y,z,x_3,\ldots,x_k$ of order $\geq 3$. Then there exists an analytic automorphism $H$ such that
\[
H(\Phi)=yz-v(x),
\] where $v(x)$ is an analytic series in $x_3,\ldots,x_k$ of order $\geq 3$. Here the equality is understood up to cyclic permutations of words.
\end{prop}
\begin{proof}
By the trick at the end of the proof of Proposition \ref{anal-inverse}, it suffices to prove the proposition for the  $k$-loop quiver. We may write
\[
F(y,z,x)=yf(y,z,x)+zg(y,z,x)+u(x)
\] where $f,g,u$ are analytic series and the equality holds up to cyclic permutation of words.
This can be done as follows. It is clear that $u(x)$ is the sum of components of $F$ that do not involve $y,z$. So every word in $F(y,z,x)-u(x)$ must contain either $y$ or $z$. We locate either $y$ or $z$ in the word  and apply a cyclic permutation so that it becomes the initial letter.  In this way,  we fix a  representative $\sum_{|w|\geq 3} c_w w$ for $F(y,z,x)-u(x)$, where the sum is over all words $w$ with initial letter being either $y$ or $z$ and $c_w=0$ if $w$ doesn't occur. {\bf{Warning:}} since $\Phi$ is defined up to cyclic permutation of words the splitting of $F$ is not unique but depends on choices of representatives. We write
\[
f=\sum_{u,|u|\geq 2} a_u u,~~~~g=\sum_{u,|u|\geq 2} b_u u.
\]  
By analyticity, we have $|a_u|\leq C^{|u|}$ and $|b_u|\leq C^{|u|}$ for some constant $C>0$.
Define for $d>1$
\[
[d]f:=\sum_{|u|=d}([u]f) u,
\] and similarly for $[d]g$. 

Let $H_0$ be the identity automorphism and let $\Phi_0=\Phi$, 
$f_0=f$, $g_0=g$ and $u_0=u(x)$. And let
\[
H_1(y)=y+[2]g,~~~H_1(z)=z+[2]f,~~~H_1(x_i)=x_i,~~~i=3,\ldots,k.
\]
Now set $\Phi_1=H_1(\Phi_0)$.  By the algorithm described above, we choose a decomposition
\[
\Phi_1=yz-yf_1(y,z,x)-g_1(y,z,x)z-u_1(x).
\] Recursively, we define
\[
H_{i+1}(y)=H_i(y)+[i+2]g_i, ~~~H_{i+1}(z)=H_i(z)+[i+2]f_i, ~~~H_1(x_i)=x_i,~~~i=3,\ldots,k,
\] and write
\[
\Phi_{i+1}:=H_i(\Phi)=yz-yf_{i+1}(y,z,x)-g_{i+1}(y,z,x)z-u_{i+1}(x).
\]
Since
\[
\ord(f_{i+1})>\ord(f_i),~~ \ord(g_{i+1})>\ord(g_i),~~ \ord(u_{i+1})\geq 3,
\] the sequence of automorphisms $H_i$ converges to a formal automorphism $H$, where 
$H(\Phi)=yz-v(x)$ holds up to cyclic permutation of words. 

It is easy to check that $H$ satisfies  the following properties
\begin{enumerate}
\item[$(1)$] For a fixed word $w$, $[w]H(y)$ and $[w]H(z)$ are polynomials in $a_u,b_u$ with $|u|\leq |w|+1$.
\item[$(2)$] We denote the above polynomials by $P_{w,y}$ and $P_{w,z}$ respectively. They both have positive coefficients.
\item[$(3)$] The degrees of $P_{w,y}$ and $P_{w,z}$ are no bigger than $2|w|+2$.
\end{enumerate}

By properties $(2)$ and $(3)$, 
\[
\left|P_{w,y}\right|\leq \left|P_{w,y}(a_u=b_u=1)\right|\cdot C^{2|w|+2}.
\] A similar estimate holds for $P_{w,z}$ as well. Now it suffices to prove the proposition assuming that $f$ and $g$ are the sums of all words of length $\geq 2$ 
(of coefficient 1).

First we focus on the case where $Q$ is a two loop quiver, i.e. $Q_1=\{y,z\}$, and set
\[
\Phi:=yz-\sum_{w,|w|\geq 3}w.
\] 
In this case, the above mentioned automorphism $H$ admits an alternative description. First note that 
\[
\Phi=yz-\sum_{m\geq 3} (y+z)^m.
\]
We may verify that
\[
H(y)=y+\sum_{m\geq 2} N_m(y+z)^m, ~~~H(z)=z+\sum_{m\geq 2} N_m(y+z)^m
\] for fixed constants $N_m$. The equation
\[
H(\Phi)=yz,
\] determines the coefficients $N_m$ uniquely.
We can easily compute the first few terms:
\[
N_2=1,~~~~N_3=6,~~~~ N_4=45.
\]
To study the growth of $N_m$, we set $G(t)=\sum_{m\geq 2} N_m t^m$. Then we have the equality
\[
(y+G(y+z))(z+G(y+z))-\frac{(y+z+2G(y+z))^3}{1-(y+z+2G(y+z))}=yz
\] 
which holds up to cyclic permutation of words. As a consequence, 
$G(t)$ satisfies the cubic equation
\[
10G^3+(15t-1)G^2+(7t^2-t)G+t^3=0.
\]
Set $G=t^2\ol{G}$. The above equation is equivalent to
\[
10t^3\ol{G}^3+(15t^2-t)\ol{G}^2+(7t-1)\ol{G}+1=0.
\]
By the implicit function theorem, the equation has an analytic solution $\ol{G}$ near $t=0$. As a consequence, $N_m$ is bounded by a geometric series.

Now we deal with the case where $Q$ is a $k$-loop quiver for arbitrary $k\geq 2$. In this case, set
\[
\Phi:=yz-(y+z)\sum_{m\geq 3} (y+z+x_3+\ldots+x_k)^{m-1}.
\]
We have
\[
H(y)=y+\sum_{m\geq 2} N_m(y+z+x_3+\ldots+x_k)^m, ~~~H(z)=z+\sum_{m\geq 2} N_m(y+z+x_3+\ldots+x_k)^m
\] such that 
\[
H(\Phi)=yz-v(x).
\]
Then $N_m$ is bounded by a geometric series by a similar argument.
\end{proof}

\begin{remark} \label{rk-analytic}
Proposition \ref{separation} allows us to construct an analytic representative
of the reduction of an analytic potential  (cf. the proof of Proposition 
\ref{anal-mutation} below). 
However, the reduction involves the 
non canonical choice of a splitting so that the {\em analytic} right equivalence
class of this representative might be non unique. 
Its {\em formal} right equivalence class
is nevertheless uniquely determined 
thanks to the following Lemma. 
\end{remark}

\begin{lemma}\cite[Proposition 4.9]{DWZ}\label{lem-mut-formaleq}
Let $(Q,\Phi)$ and $(Q,\Psi)$ be two quivers with reduced formal potentials, and $(C,T)$ be a quiver with trivial potential. If $(Q\oplus C, \Phi+T)$ is (formally) right equivalent to $(Q\oplus C, \Psi+T)$ then $(Q,\Phi)$ is (formally) right equivalent to $(Q,\Psi)$.
\end{lemma}

\begin{prop}\label{anal-mutation}
Let $Q$ be a finite quiver without loops and 2-cycles. Let $\Phi$ be an analytic potential with $\ord(\Phi)\geq 3$. Given $j\in Q_0$, the formal right equivalence class of the mutated potential $\mu_j(\Phi)$ contains an analytic representative.
\end{prop}
\begin{proof}
The pre-mutation of $\Phi$ is 
\[
\wt{\mu}_j\Phi=\underline{\Phi}+\sum_{t(\beta)=s(\alpha)=j} \beta^*\alpha^*[\alpha\beta],
\] where $\underline{\Phi}$ is obtained from $\Phi$ by replacing $\alpha\beta$ by the new arrow $[\alpha\beta]$ in $\wt{\mu}_jQ$.
Let $\Phi=\sum a_w w$ such that $|a_w|\leq C^{|w|}$.
Then
\[
|w|< 2|\underline{w}|.
\] The inequality is strict since $Q$ has no 2-cycles. Thus we have
\[
\left|[u]\underline{\Phi}\right|< C^{2|u|}.
\]
Since the second component of $\wt{\mu_j}\Phi$ is finite, it follows that $\wt{\mu_j}\Phi$ is also analytic.

Now $\wt{\mu_j}\Phi$ contains finitely many 2-cycles $\{y_1z_1,\ldots,y_dz_d\}$. We apply Proposition \ref{separation} iteratively and show that $\wt{\mu_j}\Phi$ is analytically right equivalent to 
\[
\sum_{i=1}^d y_iz_i+ \mu_j\Phi
\] where $\mu_j\Phi$ contains no arrows $\{y_1,z_1,\ldots,y_d,z_d\}$. Thus $\mu_j\Phi$ is analytic.
\end{proof}

\begin{remark} For the reason explained in Remark~\ref{rk-analytic}, we do not
know whether the analytic representative in the above Proposition is unique
up to analytic equivalence. For $\wt{J}$-finite potential, the answer is positive due to the following proposition.
\end{remark}

\begin{prop}
Let $(Q,\Phi)$ and $(Q,\Psi)$ be two quivers with reduced $\wt{J}$-finite analytic potentials, and $(C,T)$ be a quiver with trivial potential. If $(Q\oplus C, \Phi+T)$ is analytically right equivalent to $(Q\oplus C, \Psi+T)$ then $(Q,\Phi)$ is analytically right equivalent to $(Q,\Psi)$.
\end{prop}
\begin{proof}
Suppose $H$ is an analytic automorphism of $\wt{\CC Q\op T}$ such that $H(\Phi+T)=\Psi+T$.  It restricts to an analytic automorphism $H_Q$ of $\wt{\CC Q}$ such that $H_Q(\Phi)-\Psi \in \pi(\wt{J}(Q,H(\Phi))^2)$ where $\pi: \wt{\CC Q}\to \wt{\CC Q}_\cy$ is the natural projection. Without loss of generality, we may simply assume that $\Psi-\Phi\in \pi(\wt{J}(Q,\Phi)^2)$ to begin with. We claim that $\wt{J}(Q,\Phi)=\wt{J}(Q,\Psi)$. It is obvious that $\wt{J}(Q,\Psi)\subset \wt{J}(Q,\Phi)$. To show the opposite direction, note that
\[
\wt{J}(Q,\Phi)\subset \wt{J}(Q,\Psi)+ \wt{\fm} \wt{J}(Q,\Phi)+ \wt{J}(Q,\Phi)\wt{\fm}
\] since $\wt{J}(Q,\Phi)\subset \wt{\fm}^3$. Iterating the above inclusion, we obtain that 
\[
\wt{J}(Q,\Phi)\subset \wt{J}(Q,\Psi)+ \wt{\fm}^n
\] for any $n\geq 3$. Since $\Psi$ is $\wt{J}$-finite, $\wt{J}(Q,\Psi)$ is $\wt{\fm}$-adic closed in $\wt{\CC Q}$. So we conclude that $\wt{J}(Q,\Phi)= \wt{J}(Q,\Psi)$. Let $\Theta_t:=\Phi+t(\Psi-\Phi)$ be the linear family connecting $\Phi$ and $\Psi$. Then the proof of Theorem \ref{finite-determinacy} and Theorem \ref{ncMY} shows that $\Phi \sim_a \Psi$. 
\end{proof}

Thanks to the Artin approximation theorem, even though $\Phi$ and $\Psi$ are not comparable as analytic potentials in general, we are able to compare the analytic Chern-Simons functionals $\Phi_\bv$ and $\Psi_\bv$.
\begin{prop}\label{prop:welldefCS}
Let $Q$ be a finite quiver and $\Phi, \Psi$ be two analytic potentials. Suppose that $\Phi$ is formally right equivalent to $\Psi$. Given any dimension vector $\bv$, the analytic Chern-Simons functionals $\Phi_\bv$ and $\Psi_\bv$ are analytically right equivalent in a neighborhood of $\Rep_\bv^{np}(Q)$.
\end{prop}
\begin{proof}
By Lemma \ref{Toda}, we may assume that $\Phi_\bv$ and $\Psi_\bv$ are $G_\bv$-invariant analytic functions in an open neighborhood of the moduli space of nilpotent representations $\Rep_\bv^{np}(Q)$. Let $H$ be an automorphism of $\wh{\CC Q}$ such that $H(\Phi)=\Psi$. It induces a $G_\bv$-equivariant automorphism $H_\bv$ of the formal completion of $\Rep_\bv$ at the semi-simple representation $\bigoplus_{i\in Q_0} S_i^{\op v_i}$, such that $H_\bv(\Phi_\bv)=\Psi_\bv$. We consider the equation 
\[
\Psi_\bv(\bx)-\Phi_\bv({\bf{y}})=0
\]
on the affine space $\Rep_\bv(Q)\times \Rep_\bv(Q)$ with coordinates $\bx$ and ${\bf{y}}$ respectively. A formal solution ${\bf{y}}={\bf{y}}(\bx)$
can be interpreted as a formal endomorphism of $\Rep_\bv(Q)$, completed at $\bigoplus_{i\in Q_0} S_i^{\op v_i}$, that sends $\Phi_\bv$ to $\Psi_\bv$.  
Since $G_\bv$ is reductive by the equivariant version of Artin's approximation theorem (Theorem A of \cite{BM79}), $H_\bv$ is well approximated by a $G_\bv$-equivariant analytic solution. Then by Proposition \ref{anal-inverse}, $\Phi_\bv$ and $\Psi_\bv$ are analytically right equivalent. 
\end{proof}

\begin{prop}\label{nondeg-anal}
Fix a finite quiver $Q$  and $C>0$. There exists a nondegenerate analytic potential of convergence radius $1/C$.
\end{prop}
\begin{proof}
Recall an analytic potential $\Phi=\sum_c a_c c$ is of convergence radius $1/C$ if there exists $0<C_1<C$ such that $|a_c|\leq C_1^{|c|}$ for $|c|\gg 0$. The space of such potentials is denoted by $\wt{\CC Q}_{\cy,C}$. Suppose that $k=|Q_1|$. 
We claim that the norm 
\[
||\Phi||_C:= \sum_c |a_c|\left(\frac{1}{kC}\right)^{|c|}
\] makes $\wt{\CC Q}_{\cy,C}$ into a Banach space. 
We will prove the $C=1/k$ case. The general case is similar. 
Let $\Phi^1,\Phi^2,\ldots$ be a Cauchy sequence of potentials in $\wt{\CC Q}_{\cy,1/k}$. We write
\[
\Phi^i=\sum_w a_w^i w.
\]
For every $\ep>0$, there exists $N>0$ such that for $n,m> N$
\[
||\Phi^n-\Phi^m||_{1/k}=\sum_w |a_w^n-a_w^m|<\ep
\]
It follows that for any fixed $w$, $\{a_w^n\}_n$ form a Cauchy sequence. Since $\CC$ is complete, $\lim_{n\to \infty} a^n_w=a^\infty_w$ exists. Denote by $\Phi^\infty$ the formal series $\sum_w a_w^\infty w$.
Again by the Cauchy sequence property, 
\[
\sum_{|w|\leq p} |a_w^n-a_w^m|<\ep
\] for any $p>0$.
By taking $m\to \infty$, we have $\sum_{|w|\leq p} |a_w^n-a_w^\infty|<\ep$. Take $p\to \infty$, we show that 
\[
||\Phi^n-\Phi^\infty||_{1/k}<\ep
\] if $n>N$. So we prove that $\Phi^n$ converges to $\Phi^\infty$. By the triangle inequality, $||\Phi^\infty||_{1/k}<+\infty$, i.e. $\Phi^\infty\in \wt{\CC Q}_{\cy, 1/k}$.

Since being nondegenerate  is a property of the underlying formal potential, by Corollary 7.4 of \cite{DWZ} the set of nondegenerate elements in $\wt{\CC Q}_{\cy,C}$ is the complement of countably many hypersurfaces. Since a hypersurface is nowhere dense, by the 
Baire category theorem, the set of nondegenerate elements in $\wt{\CC Q}_{\cy,C}$ 
is nonempty.
\end{proof}

\begin{corollary}\label{perv-mut}
Let $Q$ be a finite quiver without loops and 2-cycles and $\Phi$ a nondegenerate analytic potential. Then there exists a canonical perverse sheaf of vanishing cycles defined on the moduli stack of finite dimensional modules on  the formal Jacobi algebra of any iterated mutation of $(Q,\Phi)$.
\end{corollary}
\begin{proof}
For the initial quiver with potential $(Q,\Phi)$, there exists a canonical perverse sheaf of vanishing cycles by Corollary \ref{pervsheaf-QP}.  By Proposition \ref{anal-mutation}, the mutation of an analytic potential $\Phi$ is analytic. Let $\Phi$ and $\Phi^\p$ be analytic potentials that are analytically right equivalent. Their mutations $\mu_j\Phi$ and $\mu_j\Phi^\p$ are formally equivalent by Proposition \ref{lem-mut-formaleq}. Therefore, the formal right equivalence class of $\mu_j\Phi$ is independent of the choice of a splitting in Proposition \ref{separation}, which also implies that there exists an analytic representative in the equivalence class. By Proposition \ref{prop:welldefCS}, all analytic representatives in the formal equivalence class of $\mu_j\Phi$ define analytically equivalent Chern-Simons functions. Therefore, the perverse sheaf of vanishing cycles is canonically defined on the moduli stack of finite dimensional modules over $\wh{\Lm}(\mu_j Q,\mu_j\Phi)$. 
\end{proof}

\subsection{Transformation of DT invariants under mutation}
In this section, we prove the transformation formula of DT invariants (weighted by Behrend function) under iterated mutations of quiver with analytic potential (Theorem \ref{thm:Nagao}). Our argument is mostly a repetition of the argument of Nagao except that we use a different integration map.

We fix a finite quiver $Q$ with $n$ nodes without loops and 2-cycles, and a nondegenerate analytic potential $\Phi$. Denote by $\wh{\Lm}$ the formal Jacobi algebra $\wh{\Lm}(Q,\Phi)$. For simplicity, we denote the abelian category 
$\mbox{Mod}-\wh{\Lm}$ 
of all pseudocompact $\wh{\Lm}$-modules by $\ol{\cA}$ and its abelian subcategory 
$\mod_{fd}-\wh{\Lm}$ of finite-dimensional modules by $\cA$. For $i\in Q_0$, denote by $P_i$ the projective indecomposable $\wh{\Lm} e_i$. For $\bv\in \NN^n$, the \emph{noncommutative Hilbert scheme} is defined to be
\[
\Hilb_{\wh{\Lm}}(i; \bv):=\{ P_i\twoheadrightarrow V| V\in \mod_\bv-\wh{\Lm}\}.
\]
It is equipped with a natural stack morphism
\[
\Hilb_{\wh{\Lm}}(i; \bv)\to \mod_\bv-\wh{\Lm}
\] 
given by forgetting the map from $P_i$. Define the \emph{generating series of  
DT invariants of the noncommutative Hilbert scheme} by
\[
Z^i_{\wh{\Lm},\nu}:=\sum_{\bv\in\NN^n}I^\nu([\Hilb_{\wh{\Lm}}(i; \bv)\to \mod_\bv-\wh{\Lm}]).
\] 
We define an algebra automorphism $DT_{\wh{\Lm}}^{\nu}$ of an appropriate completion $\wh{\TT}_\cA$ (see definition in Section 5.2.2 of \cite{Nagao}) of the semiclassical limit of the quantum double torus $\TT_Q$ by
\[
DT_{\wh{\Lm}}^{\nu}(x_i):=x_i\cdot Z^i_{\wh{\Lm},\nu}, ~~~DT_{\wh{\Lm}}^{\nu}(y_i):=y_i\cdot \prod_j\left(Z^j_{\wh{\Lm},\nu}\right)^{\ol{\chi}(j,i)}
\] where
\[
\ol{\chi}(j,i):=\chi(j,i)-\chi(i,j),~~~\chi(i,j)=\#\{a\in Q_1|s(a)=i,t(a)=j\}.
\]
The simple modules $s_i$ form a basis of the numerical Grothendieck group of $\D_{fd}\Gamma$, and $\ol{\chi}$ coincides with its euler pairing.

For a sequence of vertices $\bk=(k_1,\ldots,k_l)\in Q_0^l$, let $\mu_\bk(Q,\Phi)$ denote the iterated mutation $\mu_{k_l}\circ\mu_{k_{l-1}}\ldots\circ\mu_1(Q,\Phi)$, and let $\wh{\Lm}_\bk$ denote the formal Jacobi algebra associated to $\mu_\bk(Q,\Phi)$. Denote by $\Gamma_{(Q,\Phi)}$ the Ginzburg dg-algebra of $(Q,\Phi)$. Keller and Yang \cite{KY09} and Nagao \cite{Nagao} prove that there exists a canonical
derived equivalence
\[
\psi_\bk: \D\Gamma_{(Q,\Phi)} \iso \D\Gamma_{\mu_\bk(Q,\Phi)}. 
\] 
such that the image of the canonical heart of $\D\Gamma_{(Q,\Phi)}$ consists
of objects with homology only in degrees $-1$ and $0$ and thus the preimage
of the canonical heart of $\D\Gamma_{\mu_\bk(Q,\Phi)}$ only of objects
with homology concentrated in degrees $0$ and $1$.
Set $\ol{\cA}_\bk:=\psi_\bk^{-1}(\mbox{Mod}-\wh{\Lm}_\bk)$. There exists a torsion pair 
$(\ol{\cT}_\bk, \ol{\cF}_\bk)$ of $\ol{\cA}$ such that $\ol{\cA}_\bk$ 
is the right tilt of  $\ol{\cA}$ with respect to $(\ol{\cT}_\bk, \ol{\cF}_\bk)$
 (see Theorem 3.5 of \cite{Nagao}). Set
\[
\cA_\bk:=\psi_\bk^{-1}(\mod_{fd}-\wh{\Lm}_\bk),~~~\cT_\bk:=\ol{\cT}_\bk\cap \cA,~~~\cF_\bk:=\ol{\cF}_\bk\cap \cA.
\]
For $i\in Q_0$, consider the rigid module $R_{\bk,i}\in \cT_\bk$ defined by
\[
R_{\bk,i}:=H^1_{\ol{\cA}}(\psi_\bk^{-1}(\Gamma_{\mu_\bk(Q,\Phi)} e_i)).
\]
The \emph{quiver Grassmannian} is defined to be
\[
\Grass(\bk;i,\bv):=\{R_{\bk,i}\twoheadrightarrow V| V\in \mod_\bv-\wh{\Lm}\}.
\] It has a natural stack morphism to $\mod_\bv-\wh{\Lm}$. Let $\Sigma$ be the involution of $\TT_{\mu_\bk Q}$ that 
\[
\Sigma: x_i\leftrightarrow x_i^{-1},~~~~y_i\leftrightarrow y_i^{-1}.
\]
We define an automorphism of the double torus $\TT_{\mu_\bk Q}$ by 
\[
\Ad^\nu_{\cT_\bk[-1]}(x_{\bk,i}):=x_{\bk,i}\cdot \left(\sum_{\bv\in\NN^n}\chi(\Grass(\bk;i,\bv),\nu)\cdot \by^{-\bv}\right), 
\]
\[
\Ad^\nu_{\cT_\bk[-1]}(y_{\bk,i}):=y_{\bk,i}\cdot \prod_j\left(\sum_{\bv\in\NN^n} \chi(\Grass(\bk;i,\bv),\nu)\cdot \by^{-\bv}\right)^{\ol{\chi}(j,i)}
\] where $x_{\bk,i}$ and $y_{\bk,i}$ correspond to the K-theory classes of $\Gamma_{\mu_\bk(Q,\Phi)} e_i$ and the simple $s_{\bk,i}$ (that corresponds to the $i$-th node of $\mu_\bk Q$) respectively. We set
\[
\Ad^\nu_{\cT_\bk}:=\Sigma\circ \Ad^\nu_{\cT_{\bk[-1]}}\circ\Sigma.
\]

Let $\cM_\cA$ be the element of $\MH_{\wh{\Lm}}$ corresponding to the identity 
map $\id: \mod-\wh{\Lm}\to \mod-\wh{\Lm}$. We have $\cM_\cA=\sum_\bv \cM_\cA(\bv)$ for $\cM_\cA(\bv):=\{\id:\mod_\bv-\wh{\Lm}\to \mod_\bv-\wh{\Lm}\}$.
For $\bw$ in the Grothendieck group of the perfect derived category of $\Gamma_{Q,\Phi}$, set
\[
\cM_\cA[\bw]:=\sum_\bv \LL^{\langle\bw,\bv\rangle} \cdot \cM_\cA(\bv).
\]
We denote by $\bw_i$ the class of $\Gamma_{(Q,\Phi)} e_i$.
\begin{prop}\label{Nagao8.18}
\begin{align}\label{Nagao8.18a}
\Ad^\nu_\cA(x_i)=x_i\cdot I^\nu\left((\cM_\cA[\bw_i] \star \cM_\cA^{-1})|_{\LL=1}\right),
\end{align}
\begin{align}\label{Nagao8.18b}
\Ad^\nu_\cA(y_i)=y_i\cdot \prod_{j} I^\nu\left((\cM_\cA[\bw_j] \star \cM_\cA^{-1})|_{\LL=1}\right)^{\ol{\chi}(j,i)}.
\end{align}
\end{prop}
\begin{proof}
Here we follow the notation of \cite{Nagao}. Recall from Section 7.3 of \cite{Nagao} that the product on the double torus ${\rm{Q}}\TT:=\QT^\vee\ot_{\CC(t)} \QT$ is  defined by:
\[
\by^\bv\cdot \by^{\bv^\p}=t^{\ol{\chi}(\bv^\p,\bv)} \by^{\bv+\bv^\p},~~~ \bx^{\bw}\cdot \bx^{\bw^\p}=\bx^{\bw+\bw^\p},~~~ y_i\cdot x_j=t^{2\delta_{ij}} x_j\cdot y_i.
\]


The Poisson algebra structure on ${\rm{Q}}\TT^\nu_{sc}$ is given by
\[
\by^\bv\cdot \by^{\bv^\p}=(-1)^{\ol{\chi}(\bv^\p,\bv)} \by^{\bv+\bv^\p},~~~ \bx^{\bw}\cdot \bx^{\bw^\p}=\bx^{\bw+\bw^\p},~~~ y_i\cdot x_j= x_j\cdot y_i
\] and
\[
\{\by^\bv, \by^{\bv^\p}\}=(-1)^{\ol{\chi}(\bv^\p,\bv)}\ol{\chi}(\bv^\p,\bv)\cdot \by^{\bv+\bv^\p},~~~ \{\bx^{\bw}, \bx^{\bw^\p}\}=0,~~~ \{y_i, x_j\}=  x_j y_i=y_ix_j.
\]
This extends the Poisson commutative algebra structure on $\QT_{sc}$ defined in Theorem \ref{integration}.  

Put
\[
\ep_\cA:=\log(1+\cM_\cA):=\sum_{l\geq 1}\frac{(-1)^l}{l} \cM_\cA\star\ldots\star\cM_\cA,
\] as an element of appropriate completion of $\wh{\MH}_{\wh{\Lm}}$.
Then we have  $\cM_\cA=\exp(\ep_\cA)$. Set $\ep_\cA=\sum_\bv \ep_\cA(\bv)$. By the no pole theorem (cf. Theorem 7.3 \cite{Nagao}), $\wt{\ep}_\cA:=(\LL-1)\cdot \ep_\cA$ is regular. We set
$\wh{\ep}_\cA=\wt{\ep}_\cA|_{\LL=1}$ and then $\wh{\ep}_\cA\in \wh{\MH}_{\wh{\Lm},sc}$. Set
\[
\ep_\cA[\bw_i]:=\sum\LL^{\lg [\bw_i],\bv\rg}\cdot \ep_\cA(\bv),~~~~~\wh{\ep}_\cA\{\bw_i\}:=\sum \lg \bw_i,\bv\rg \cdot \wh{\ep}_\cA(\bv).
\]
By Lemma 8.15 of \cite{Nagao}, $\ep_\cA[\bw_i]-\ep_\cA$ is regular and 
\[
(\ep_\cA[\bw_i]-\ep_\cA)|_{\LL=1}=\wh{\ep}_\cA\{\bw_i\}.
\] We claim that 
\begin{align}\label{w_i}
\{ I^\nu(\wh{\ep}_\cA),x_i\}= I^\nu(\wh{\ep}_\cA\{\bw_i\})\cdot x_i.
\end{align} 
This follows from a straighforward calculation:
\begin{align*}
\{ I^\nu(\wh{\ep}_\cA),x_i\}&=\sum \chi(\wh{\ep}_\cA(\bv), \nu) \{\by^\bv, x_i\}= \left(\sum_\bv  v_i\chi(\wh{\ep}_\cA(\bv), \nu) \cdot \by^\bv\right)\cdot x_i\\
&=I^\nu\left(\sum_\bv   v_i\wh{\ep}(\bv)\right)\cdot x_i= I^\nu\left(\sum_\bv \lg \bw_i, \bv\rg\cdot\wh{\ep}(\bv)\right)\cdot x_i= I^\nu(\wh{\ep}_\cA\{\bw_i\})\cdot x_i.
\end{align*}
Set 
\[
\cE^{\lg p\rg}_{i,\cA}:=\sum_j (-1)^j{p\choose j}\ep_\cA[\bw_i]^{\star (p-j)}\star \ep_\cA^{\star (j)} \in \wh{\MH}_{\wh{\Lm}}.
\]
Then 
\[
\cM_\cA[\bw_i]\star \cM_\cA^{-1}=\sum \frac{1}{p!}\cdot \cE^{\lg p\rg}_{i,\cA}.
\]
By Lemma 8.16 of \cite{Nagao}, we have
\begin{align}\label{Nagao8.16}
\cE^{\lg p+1\rg}_{i,\cA}=(\ep_\cA[\bw_i]-\ep_\cA)\cdot \cE^{\lg p\rg}_{i,\cA}+\frac{1}{\LL-1} [\wt{\ep}_\cA,\cE^{\lg p\rg}_{i,\cA}],
\end{align} therefore $\cE^{\lg p\rg}_{i,\cA}$ is regular for all $p\geq 0$.
To verify formula \ref{Nagao8.18a}, it suffices to check that 
\[
\left(\{I^\nu(\wh{\ep}_\cA),?\}\right)^p(x_i)=I^\nu(\cE^{\lg p\rg}_{i,\cA}|_{\LL=1})\cdot x_i.
\]
We will prove it by induction on $p$.
When $p=0$, the equality holds trivially. 
\begin{align*}
\{I^\nu(\wh{\ep}_\cA), I^\nu(\cE^{\lg p\rg}_{i,\cA}|_{\LL=1})\cdot x_i\}&=\left(I^\nu(\cE^{\lg p\rg}_{i,\cA}|_{\LL=1})\cdot  I^\nu(\wh{\ep}_\cA\{\bw_i\})+\{I^\nu(\wh{\ep}_\cA), I^\nu(\cE^{\lg p\rg}_{i,\cA}|_{\LL=1})\}\right)\cdot x_i\\
&=I^\nu\left(\cE^{\lg p\rg}_{i,\cA}|_{\LL=1}\cdot \wh{\ep}_\cA\{\bw_i\}+\{\wh{\ep}_\cA,\cE^{\lg p\rg}_{i,\cA}|_{\LL=1}\}\right)\cdot x_i\\
&=I^\nu(\cE^{\lg p+1\rg}_{i,\cA}|_{\LL=1})\cdot x_i.
\end{align*}
The first equality follows from the Leibniz rule and the identity \ref{w_i}. The second equality follows from the fact that $I^\nu$ is a Poisson algebra morphism. The last equality follows from identity \ref{Nagao8.16}  and the fact that $\wh{\MH}_{\wh{\Lm},sc}$ is commutative.

Now we prove the second identity. We set
\[
\cF_{i,\cA}^{\lg p\rg}:=\sum_l(-1)^l {p\choose l} (\ep_\cA[\sum_j \ol{\chi}(j,i)\cdot \bw_j])^{\star (p-l)}\cdot \ep_\cA^{\star(l)}
\] Then
\[
\cM_\cA[\sum_j\ol{\chi}(j,i)\cdot \bw_j]\star \cM_\cA^{-1}=\sum_p \frac{1}{p!} \cF_{i,\cA}^{\lg p\rg}.
\]
Moreover $\cF_{i,\cA}^{\lg p\rg}$ satisfies the recursive relation:
\[
\cF^{\lg p+1\rg}_{i,\cA}=(\ep_\cA[\sum_j\ol{\chi}(j,i) \bw_i]-\ep_\cA)\cdot \cF^{\lg p\rg}_{i,\cA}+\frac{1}{\LL-1} [\wt{\ep}_{\cA}, \cF^{\lg p\rg}_{i,\cA}],
\] and in particular it is regular.
Since we have (see proof of 8.18 \cite{Nagao}):
\begin{align}\label{addition-frame}
\cM_{\cA}[\bw+\bw^\p]\star \cM_\cA^{-1}=(\cM_\cA[\bw]\star \cM_\cA^{-1})[\bw^\p]\star (\cM_\cA[\bw^\p]\star \cM_\cA^{-1}),
\end{align}
the identity in $\wh{\MH}_{\wh{\Lm},sc}$
\[
\left(\cM_\cA[\sum_j\ol{\chi}(j,i) \bw_j]\star \cM_\cA^{-1}\right)\Bigg|_{\LL=1}=\prod_j\left(\cM_\cA[\bw_j]\star \cM_\cA^{-1}\Big|_{\LL=1}\right)^{\ol{\chi}(j,i)}
\] holds. It suffices to show that 
\[
(\{ I^\nu(\wh{\ep}_\cA),?\})^p(y_i)=I^\nu(\cF^{\lg p\rg}_{i,\cA}|_{\LL=1}) y_i.
\]
When $p=0$, this holds trivially. We will prove it by induction. 
\begin{align*}
\{I^\nu(\wh{\ep}_\cA), I^\nu(\cF^{\lg p\rg}_{i,\cA}|_{\LL=1})y_i\} &=\{I^\nu(\wh{\ep}_\cA), y_i\}\cdot I^\nu(\cF^{\lg p\rg}_{i,\cA}|_{\LL=1})+\{I^\nu(\wh{\ep}_\cA),  I^\nu(\cF^{\lg p\rg}_{i,\cA}|_{\LL=1})\} y_i\\
&=\left( I^\nu(\wh{\ep}_\cA\{\sum_j \ol{\chi}(j,i)\bw_j\}\cdot I^\nu(\cF^{\lg p\rg}_{i,\cA}|_{\LL=1})+\{I^\nu(\wh{\ep}_\cA),  I^\nu(\cF^{\lg p\rg}_{i,\cA}|_{\LL=1})\}\right) y_i\\
&=I^\nu\left( \wh{\ep}_\cA\{\sum_j \ol{\chi}(j,i)\bw_j\}\cdot \cF^{\lg p\rg}_{i,\cA}|_{\LL=1}+\{ \wh{\ep}_\cA,   \cF^{\lg p\rg}_{i,\cA}|_{\LL=1})\}\right) y_i\\
&=I^\nu(\cF^{\lg p+1\rg}_{i,\cA}|_{\LL=1}) y_i
\end{align*}
The second equality follows from the definition of the Poisson bracket on $\mathrm{Q}\TT_{sc}$ and the last equality follows from the recursive relation for 
$\cF^{\lg p+1\rg}_{i,\cA}$.
\end{proof}
 
The proposition has two important corollaries.
\begin{theorem}(cf. \cite[Theorem 5.7]{Nagao})\label{thm:Nagao}
We have the following identity of automorphisms of $\wh{\TT}_{\cA_\bk}$:
\[
DT_{\wh{\Lm}_\bk}^{\nu}=(\Ad^\nu_{\cT_\bk})^{-1}\circ DT_{\wh{\Lm}}^{\nu}\circ \Ad^\nu_{\cT_\bk[-1]}.
\]
\end{theorem}
\begin{proof}
By Proposition \ref{separation}, $\mu_\bk\Phi$ is formally right equivalent to an analytic potential. According to Lemma  \ref{lem-mut-formaleq}, the formal right equivalence class of $\mu_\bk\Phi$ is uniquely determined. If we fix an analytic representative then by Proposition \ref{prop:welldefCS}, the analytic right equivalence class of its associate Chern-Simons functional is uniquely determined. Therefore $DT_{\wh{\Lm}_\bk}^{\nu}$ is well defined.  

By the motivic Hilbert scheme identity (Proposition 8.1 \cite{Nagao}) and Proposition \ref{Nagao8.18},  we have
\[
\Ad^\nu_{\cA}(x_i)=x_i\cdot \left(\sum_\bv \chi(\Hilb_{\wh{\Lm}}(i,\bv),\nu) \by^\bv\right)=DT^\nu_{\wh{\Lm},\nu}(x_i)
\] and
\[
\Ad^\nu_{\cA}(y_i)=y_i\cdot \prod_j\left(\sum_\bv \chi(\Hilb_{\wh{\Lm}}(j,\bv),\nu) \by^\bv\right)^{\ol{\chi}(j,i)}=DT^\nu_{\wh{\Lm},\nu}(y_i).
\]
Now replace $\cA$ by the tilted heart $\cA_{\bk}$.
Then the theorem follows from the motivic torsion pair identity (Proposition 8.2 \cite{Nagao}).
\end{proof}

\subsection{Perverse $F$-series and perverse Caldero-Chapoton formula}
In this section, we define perverse $F$-series and prove a perverse analogue of 
the Caldero--Chapton formula. It is again proved by adapting Nagao's argument. 
However, since the singularities of the moduli stack are taken into the consideration, our formula for the adjoint operator associated with a simple module differs from that in \cite{Nagao} by a sign. 
To obtain the Fomin-Zelevinsky's cluster transformation, we then need to plug 
in a different equation for $y$-variables.  

Let $Q$ be a quiver, $\Phi$ an analytic potential and $\wh{\Lm}$ its formal Jacobi algebra. Given a finite dimensional $\wh{\Lm}$-module $M$, define the quiver Grassmannian
\[
\Grass(M,\bv):=\{M\twoheadrightarrow V| V\in \mod_\bv-\wh{\Lm}\}.
\]  
The \emph{F-series} of $M$ is defined to be
\[
F(M):=\sum_{\bv} I\left([f: \Grass(M,\bv)\to \mod_\bv-\wh{\Lm}]\right)
\]
The \emph{perverse F-series} of $M$ is defined to be
\[
F^\nu(M):=\sum_{\bv} I^\nu\left([f: \Grass(M,\bv)\to \mod_\bv-\wh{\Lm}]\right).
\]
The perverse $F$-series is usually different from the ordinary F-series. We will exhibit an example.

Let $f: Y\to X$ be a 3-dimensional simple flopping contraction, which means
\begin{enumerate}
\item[$(1)$] $Y$ is a smooth quasi-projective 3-fold;
\item[$(2)$] $f$ is a birational morphism that is an isomorphism in codimension one and $K_Y$ is $f$-trivial;
\item[$(3)$] The exceptional fibers of $f$ are irreducible.
\end{enumerate}
Let $p$ be a singular point of $X$ and $\wh{R}$ be the formal completion of $X$ at $p$. Then $\wh{R}$ is a complete Noetherian hypersurface ring of Krull dimension 3, i.e. $\wh{R}\cong \CC[[x,y,z,w]]/(g)$. Denote by $\wh{f}: \wh{Y}\to \wh{X}:=\Spec \wh{R}$ the base change of $f$. Denote by $C$ the reduced fiber of $p$. The Ext-quiver $Q$ of $\cO_C$ is a $k$-loop quiver for $k=0,1,2$. The Ext-algebra of $\cO_C$ is a cyclic $A_\infty$-algebra, therefore determines a unique formal potential up to right equivalence
 by the theorem of Van den Bergh \cite{VdB15}. Its formal Jacobi algebra $\wh{\Lm}$ is finite dimensional (cf. \cite{HuaKeller}). The perverse $F$-series of $\wh{\Lm}$ has been calculated in Section 4 of \cite{HT17} using the wall crossing formula for DT invariants:
\[
F^\nu(\wh{\Lm})=\prod_{j=1}^l  \left(1-(-1)^j y^j\right)^{j n_j}
\] where $l$ is the \emph{length} of $C$ and the positive integers $n_j$ are the \emph{Gopakumar-Vafa invariants} of $Y$ with curve classes $j[C]$ (see \cite{HT17}). The length $l$ takes values in 
$\{1,2,3,4,5,6\}$ and can be computed by taking a generic hyperplane section of $\wh{X}$. The curve $C$ has length $1$ if and only if $\wh{\Lm}\cong \CC[[t]]/t^{n_1}$.  
In this case, we have $F^\nu(\wh{\Lm})=(1+y)^{n_1}$ and $n_j=0$ for $j>1$.  If $l=1$, then we have
\[
F(\wh{\Lm})=1+y+\ldots+y^{n_1},
\]
since the module $\wh{\Lm}\cong \CC[[t]]/t^{n_1}$ has exactly one quotient module  up to isomorphisms of dimension $d$ for $1\leq d\leq n_1$.

From this example, we see that even though the definition of perverse $F$-series 
only makes sense under more restrictive hypotheses than that of the ordinary $F$-series 
(e.g. it requires the potential to be analytic), the calculation of it might be easier than the ordinary $F$-series since the exponents $n_j$ are deformation invariants and therefore can be computed using  degeneration methods. 

To finish this section, we prove the perverse version of the Caldero--Chapton formula. 
We begin by recalling some notations from \cite{Nagao}. Let $\bk:=(k_1,\ldots,k_l)$ be a sequence of nodes. For $1\leq r\leq l$, denote by $\bk(r)$ the sequence $(k_1,\ldots,k_r)$. Let $s_{\bk(r),i}$ be the simple module for node $i$ in the tilted heart $\cA_{\bk(r)}$, where $\cA_{\bk(0)}=\cA$. For $i\in \{1,\ldots,n\}$, set 
\[
s_i^{(r)}:=\begin{cases}
\psi^{-1}_{\bk(r)}(s_{\bk(r),i}) & \text{if } \psi^{-1}_{\bk(r)}(s_{\bk(r),i})\in \cF_{\bk(r)}, \\
\psi^{-1}_{\bk(r)}(s_{\bk(r),i})[1] & \text{if } \psi^{-1}_{\bk(r)}(s_{\bk(r),i})\in \cT_{\bk(r)}[-1].
\end{cases}
\]
In particular, we set $s^{(r)}:=s^{(r)}_{k_r}$. Let $\cS(r)$ be the additive closure of $s^{(r)}$ in $\cA$. Apply the argument of Proposition \ref{Nagao8.18} to $\cC=\cT_\bk$ or $\cS(r)$, the identities 
\begin{align}\label{general8.18}
\Ad^\nu_\cC(x_{\bk(r),i})&=x_{\bk(r), i}\cdot I^\nu\left((\cM_\cC[\bw_{\bk(r), i}] \star \cM_\cC^{-1})|_{\LL=1}\right), \\
\Ad^\nu_\cC(y_{\bk(r),i})&=y_{\bk(r),i}\cdot \prod_{j} I^\nu\left((\cM_\cC[\bw_{\bk(r), j}] \star \cM_\cC^{-1})|_{\LL=1}\right)^{\ol{\chi}_{\bk(r)}(j,i)}
\end{align} hold, where $\ol{\chi}_{\bk(r)}(j,i)=\ol{\chi}([s_{\bk(r),j}],[s_{\bk(r),i}])$. 
For $\cC=\cS(r)$, by the motivic quiver Grassmannian identity (Proposition 8.4 \cite{Nagao}) we further have
\[
\Ad^\nu_{\cS(r)}(x_{\bk(r),i})=x_{\bk(r), i}\cdot \left(\sum_\bv\chi(\Grass(R_{\bk(r),i},\cA_{\bk(r-1)},\bv),\nu) \by^\bv\right),
\] and 
\[
\Ad^\nu_{\cS(r)}(y_{\bk(r),i})=y_{\bk(r),i}\cdot \prod_{j} \left(\sum_\bv\chi(\Grass(R_{\bk(r),j},\cA_{\bk(r-1)},\bv),\nu) \by^\bv\right)^{\ol{\chi}(j,i)}
\] where $\Grass(R_{\bk(r),i},\cS(r),\bv)$ is the stack of quotients of $R_{\bk(r),i}$ in 
$\cS(r)$ of dimension vector $\bv$.

Using the derived equivalence and the fact that
\[
R_{\bk(r),i}=\begin{cases}
0 & i\neq k_r\\
s^{(r-1)}  & i=k_r
\end{cases}
\]
we have
\begin{align*}
\Grass(R_{\bk(r),i},\cA_{\bk(r-1)},\bv)
&=\begin{cases}
\emptyset & \text{if } i\neq k_r\\
[\Spec \CC] &  \text{if } i= k_r ~~\bv=0~~\text{or}~~ [s^{(r-1)}]\\
\emptyset & \text{otherwise}
\end{cases}
\end{align*}
Since the analytic Chern-Simons function $\nu$ is zero on $\cM_{\cS(r)}$, we have
\[
\Ad^\nu_{\cS(r)}(x_{\bk(r),i})=\begin{cases}
x_{\bk(r),i} & \text{if}~ i\neq k_r\\
x_{\bk(r),i}\cdot(1-\by^{[s^{(r-1)}]}) & \text{if}~i=k_r.
\end{cases}
\]  The minus sign  is due to the existence of an one dimensional stabilizer 
(see the formula for the Behrend function of a smooth stack in the proof of 
Theorem \ref{integration}).
Recall that $x_{\bk,i}=\bx^{[\Gamma_{\bk,i}]}$ and $y_{\bk,i}=\by^{[s_{\bk,i}]}$. When $\bk(1)=(k)$ we have 
\[
x_{(k),i}=\begin{cases}
x_i & i\neq k\\
x_k^{-1}\prod_j x_j^{\chi(j,k)} & i=k.
\end{cases}
\]
If we specialize at the equations $y_i=-\prod_j x_j^{\ol{\chi}(j,i)}$, we get the \emph{Fomin--Zelevinsky exchange relation}
\[
FZ_{(k),i}=\begin{cases}
x_i & \text{if}~ i\neq k\\
x_k^{-1}(\prod_j x_j^{\chi(j,k)}+\prod_j x_j^{\chi(k,j)}) & \text{if}~i=k.
\end{cases}
\]
 
The following theorem is the perverse analogue of Caldero-Chapton formula (Theorem 5.8 \cite{Nagao}).
\begin{theorem} \label{thm:perverseCC}
Let $(\ep(1),\ldots,\ep(r))$ be the sequence of signs of
Theorem~3.5 of \cite{Nagao}. Then we have
\[
(\Ad^\nu_{\cS(1)[-1]})^{\ep(1)}\circ \ldots \circ (\Ad^\nu_{\cS(l)[-1]})^{\ep(l)}(x_{\bk,i})=x_{\bk,i}\cdot\left(\sum_\bv \chi(\Grass(\bk;i,\bv),\nu) \cdot \by^{-\bv}\right).
\]  
\end{theorem}
\begin{proof}
Recall that for each truncated sequence $\ep(1),\ldots,\ep(l)$,
we have
\[
\cM_{\cT_\bk}=(\cM_{\cS(1)})^{\ep(1)}\star\ldots\star (\cM_{\cS(l)})^{\ep(l)}
\] in $\wh{\MH}_{\wh{\Lm}}$. This is called motivic factorization identity 
(Proposition 8.8 of \cite{Nagao}).
We prove a lemma which is an analogue of Proposition 8.27 of \cite{Nagao}.
\begin{lemma}
We have that 
\[
\Ad^\nu_{\cT_\bk[-1]}=(\Ad^\nu_{\cS(1)[-1]})^{\ep(1)}\circ \ldots \circ (\Ad^\nu_{\cS(l)[-1]})^{\ep(l)},
\] or equivalently 
\[
\Ad^\nu_{\cT_\bk}=(\Ad^\nu_{\cS(1)})^{\ep(1)}\circ \ldots \circ (\Ad^\nu_{\cS(l)})^{\ep(l)}.
\]
\end{lemma}
\begin{proof}
We will prove the second identity. Set $\delta(r):=\cM_{\cS(r)}$. Then 
\begin{align*}
(\Ad^\nu_{\cS(r)})^{\ep(r)}(\bx^\bw)&=(\Ad^\nu_{\cS(r)})^{\ep(r)}(\bx^{\sum_in_{\bk,i}\bw_{\bk,i}})=\left(\prod_i(\Ad^\nu_{\cS(r)})^{n_{\bk,i}}(\bx_{\bk,i})\right)^{\ep(r)}\\
&=\left(\bx^\bw\cdot I^\nu(\prod_i(\delta(r)[\bw_{\bk,i}]\star \delta(r)^{-1})^{n_{\bk,i}}|_{\LL=1})\right)^{\ep(r)}\\
&= \bx^\bw\cdot I^\nu\left((\delta(r)[\bw])^{\ep(r)}\star \delta(r)^{-\ep(r)}|_{\LL=1} \right)
\end{align*}
The second equality follows from identity \ref{general8.18} with $\cC=\cS(r)$. The third identity uses formula \ref{addition-frame} and the fact that $I^\nu$ is a Poisson morphism.  For any $X\in \wh{\MH}_{\wh{\Lm}}$ We have an identity
\[
\Ad^\nu_{\cS(r)}(I^\nu(X|_{\LL=1}))=I^\nu((\delta(r)\star X\star \delta(r)^{-1})|_{\LL=1}),
\] whose proof is the same as that of Lemma 8.25 of \cite{Nagao}. Using this identity and induction, we can show that for $1\leq r\leq l$, 
\begin{align*}
&(\Ad^\nu_{\cS(r)})^{\ep(r)}\circ \ldots \circ (\Ad^\nu_{\cS(l)})^{\ep(l)}(\bx^\bw)\\
&=\bx^\bw\cdot \prod_{r^\p=r}^l I^\nu\left(\delta(r)^{\ep(r)}\star\ldots\star (\delta(r^\p)[\bw])^{\ep(r^\p)}\star \delta(r^\p)^{-\ep(r^\p)}\star\ldots \star \delta(r)^{-\ep(r)}\Big|_{\LL=1}\right).
\end{align*}
Now set $r=1$. We have
\begin{align*}
&(\Ad^\nu_{\cS(1)})^{\ep(1)}\circ \ldots \circ (\Ad^\nu_{\cS(l)})^{\ep(l)}(\bx^\bw)\\
&=\bx^\bw\cdot \prod_{r^\p=1}^l I^\nu\left(\delta(1)^{\ep(1)}\star\ldots\star (\delta(r^\p)[\bw])^{\ep(r^\p)}\star \delta(r^\p)^{-\ep(r^\p)}\star\ldots \star \delta(r)^{-\ep(r)}\Big|_{\LL=1}\right)\\
&=\bx^\bw\cdot I^\nu\left( (\delta(1)[\bw])^{\ep(1)}\star\ldots\star (\delta(l)[\bw])^{\ep(l)} \star \delta(l)^{-\ep(l)}\star\ldots \delta(1)^{-\ep(1)}\Big|_{\LL=1}\right)\\
&=\bx^\bw\cdot I^\nu(\cM_{\cT_\bk}[\bw_{\bk,i}]\star \cM_{\cT_\bk}^{-1}\Big|_{\LL=1})\\
&=\Ad^\nu_{\cT_\bk}(\bx^\bw).
\end{align*}
The second equality uses the fact that $\wh{\MH}_{\wh{\Lm},sc}$ is commutative. The third equality uses Proposition 8.9 of \cite{Nagao}. The last equality follows from identity \ref{general8.18}.
\end{proof}

Now we prove the theorem.
If we apply identity \ref{general8.18} with $\cC=\cT_\bk[-1]$ and the motivic quiver Grassmannian identity (Proposition 8.7 of \cite{Nagao}), we have
\begin{align*}
\Ad^\nu_{\cT_{\bk}[-1]}(x_{\bk,i})&=x_{\bk, i}\cdot \left(\sum_\bv \chi(\Grass(\bk;i,\bv),\nu) \cdot \by^{-\bv}\right), \\
\Ad^\nu_{\cT_{\bk}[-1]}(y_{\bk,i})&=y_{\bk,i}\cdot \prod_{j} \left(\sum_\bv \chi(\Grass(\bk;j,\bv),\nu) \cdot \by^{-\bv}\right)^{\ol{\chi}(j,i)}.
\end{align*}
Now the claim follows the above Lemma.
\end{proof}

\subsection{Anti-cluster algebras and the anti-cluster character}

Let $1\leq n \leq m$ be integers and $\wt{B}$ an integer $m\times n$-matrix whose
upper $n\times n$-block $B$ is antisymmetric. The matrix $\wt{B}$ corresponds
to an ice quiver $\wt{Q}$ with frozen vertices $n+1, \ldots, m$ and the
matrix $B$ to the full subquiver $Q$ of $\wt{Q}$, cf. section~4.1 of
\cite{Keller12}. Suppose that $\Lambda$ is an antisymmetric integer
$m\times m$-matrix such that
\[
\wt{B}^T \Lambda = [I_n\;\; 0].
\]
Thus, the pair $(\wt{B},\Lambda)$ is {\em compatible} in the sense of
Berenstein--Zelevinsky \cite{BerensteinZelevinsky05}. The definition
of the {\em anti-cluster algebra} $\cA^-(\wt{Q},\Lambda)$ associated with
$(\wt{Q},\Lambda)$ is obtained by specializing the parameter $q^{1/2}$ to
$-1$ in the definition in [loc. cit.] of the quantum cluster algebra associated
with $(\wt{Q},\Lambda)$. Since then $q$ is specialized to $1$, the anti-cluster
algebra is commutative but additional signs appear in the exchange relations.

Let $\Gamma$ be the (complete) dg Ginzburg algebra associated with
$\wt{Q}$ and a non degenerate analytic potential $\Phi$. Exceptionally,
we denote by $\cD(\Gamma)$ the derived category whose objects are
dg {\em left} $\Gamma$-modules and by $\per(\Gamma)$ the perfect
derived category, i.e. the thick triangulated subcategory of $\cD(\Gamma)$
generated by the free dg module $\Gamma$. The Grothendieck group
of $\per(\Gamma)$ has a basis given by the classes of the modules
$\Gamma e_i$, $1\leq i\leq m$.
Thus, if $P$ belongs to $\per(\Gamma)$ and $[P]$ has coordinates
$a_i$ in the basis of the $[\Gamma e_i]$, then we put
\[
x^{[P]}=\prod_{i=1}^m x_i^{a_i}.
\]
Recall that each dg module $N$ whose homology is of finite total dimension
belongs to $\per(\Gamma)$ so that the expression $x^{[N]}$ is well-defined.
Let $\cC$ be the full subcategory
of $\per(\Gamma)$ whose objects are the dg modules $M$ 
such that $H^1 M$ is finite-dimensional and supported on $Q \subset \wt{Q}$
and there is a triangle
\[
\Sigma^{-1} T_0 \to M \to T_1 \to T_0
\]
where $T_0$ and $T_1$ are direct summands of finite direct sums of copies of $\Gamma$. For $M$ in $\cC$, we put
\[
CC^-(M) = x^{[M]} \sum_\bv \chi(\Grass(H^1 M, \bv), \nu) \;x^{[H^1 M]}.
\]
Here, we write $\Grass(H^1 M, \bv)$ for the variety of {\em quotients} of the 
$\Gamma$-module $H^1 M$ with dimension vector $\bv$ and 
$\chi(\Grass(H^1 M, \bv), \nu)$ for its Behrend-weighted Euler characteristic. 

Let $\TT$ denote the $n$-regular tree with root $t_0$ where, at each vertex $t$,
the $n$ outgoing edges are labeled $1, \ldots, n$. 
For each vertex $t$ of $\TT$ and each $1 \leq i \leq m$, we have the 
anti-cluster variable $x_i(t)$ and the indecomposable rigid object 
$T_i(t)$ belonging to $\cC\subset\per(\Gamma)$ 
defined in section~7.7 of \cite{Keller12}.

\begin{theorem} We have $x_i(t) = CC^-(T_i(t))$.
\end{theorem}

We omit the proof since it is very similar to that of Theorem~\ref{thm:perverseCC}.

\end{document}